\documentclass[11pt, oneside]{amsart}   	
\usepackage{geometry}                		
\usepackage[parfill]{parskip}    		
\usepackage{graphicx}
\usepackage{color}
\usepackage{amssymb,amsmath}
\newtheorem{theorem}{Theorem}
\newtheorem{lemma}[theorem]{Lemma}
\newtheorem{corollary}[theorem]{Corollary}

\theoremstyle{definition}
\newtheorem{remark}[theorem]{Remark}
\newtheorem{example}[theorem]{Example}
\newtheorem{problem}[theorem]{Problem}
\newtheorem{definition}[theorem]{Definition}
\usepackage{epstopdf}
\numberwithin{equation}{section}
\numberwithin{theorem}{section}

\author{B.~Fuglede}
\address{Department of Mathematical Sciences, University of Copenhagen, Universitetsparken 5,
2100 Copenhagen, Denmark}
\email{fuglede@math.ku.dk}

\author{N.~Zorii}
\address{Institute of Mathematics,
National Academy of Sciences of Ukraine, Tereshchenkivska 3, 01601,
Kyiv-4, Ukraine}
\email{natalia.zorii@gmail.com}

\begin{document}

\title[Condensers with touching plates]
{Constrained minimum Riesz and Green energy problems for vector measures associated with a~generalized condenser}

\begin{abstract}
For a finite collection $\mathbf A=(A_i)_{i\in I}$ of locally closed sets in $\mathbb R^n$, $n\geqslant3$, with the sign $\pm1$ prescribed such that the oppositely charged plates are mutually disjoint, we consider the minimum energy problem relative to the $\alpha$-Riesz kernel $|x-y|^{\alpha-n}$, $\alpha\in(0,2]$, over positive vector Radon measures $\boldsymbol\mu=(\mu^i)_{i\in I}$ such that each $\mu^i$, $i\in I$, is carried by $A_i$ and normalized by $\mu^i(A_i)=a_i\in(0,\infty)$.  We show that, though the closures of oppositely charged plates may intersect each other even in a set of nonzero capacity, this problem has a solution $\boldsymbol\lambda^{\boldsymbol\xi}_{\mathbf A}=(\lambda^i_{\mathbf A})_{i\in I}$ (also in the presence of an external field) if we restrict ourselves to $\boldsymbol\mu$
with $\mu^i\leqslant\xi^i$, $i\in I$, where the constraint $\boldsymbol\xi=(\xi^i)_{i\in I}$ is properly chosen. We establish the sharpness of the sufficient conditions on the solvability thus obtained, provide descriptions of the weighted vector $\alpha$-Riesz potentials of the solutions, single out their characteristic properties, and analyze the supports of the $\lambda^i_{\mathbf A}$, $i\in I$. Our approach is based on the simultaneous use of the vague topology and an appropriate semimetric structure defined in terms of the $\alpha$-Riesz energy on a set of vector measures associated with $\mathbf A$, as well as on the establishment of an intimate relationship between the constrained minimum $\alpha$-Riesz energy problem and a constrained minimum $\alpha$-Green energy problem, suitably formulated. The results are illustrated by examples.
\end{abstract}
\maketitle

\section{Introduction}

The purpose of this paper is to study {\it minimum energy problems with external fields\/} (also known in the literature as {\it weighted minimum energy problems\/} or as {\it Gauss variational problems\/}) relative to the {\it $\alpha$-Riesz kernel\/} $\kappa_\alpha(x,y)=|x-y|^{\alpha-n}$ of  order $\alpha\in(0,2]$, where $|x-y|$ is the Euclidean distance between $x,y\in\mathbb R^n$, $n\geqslant3$, and infimum is taken over classes of vector measures $\boldsymbol\mu=(\mu^i)_{i\in I}$ associated with a generalized condenser $\mathbf A=(A_i)_{i\in I}$ and normalized by $\mu^i(A_i)=a_i\in(0,\infty)$, $i\in I$. More precisely, an ordered finite collection $\mathbf A$ of {\it locally closed\/} sets $A_i$, $i\in I$, termed {\it plates\/}, with the sign $s_i=\pm1$ prescribed is said to be a {\it generalized condenser\/} if the oppositely signed plates are mutually disjoint, while a vector measure $\boldsymbol\mu=(\mu^i)_{i\in I}$ is said to be {\it associated with\/} $\mathbf A$ if each $\mu^i$, $i\in I$, is a positive scalar Radon measure (charge) carried by $A_i$. In accordance with an electrostatic interpretation of a condenser, we say that the interaction between the components $\mu^i$, $i\in I$, of such a $\boldsymbol\mu$ is characterized by the matrix $(s_is_j)_{i,j\in I}$, so that the {\it $\mathbf f$-weighted $\alpha$-Riesz energy\/} of $\boldsymbol\mu$ is defined by
\[G_{\kappa_\alpha,\mathbf f}(\boldsymbol\mu):=\sum_{i,j\in I}\,s_is_j\iint|x-y|^{\alpha-n}\,d\mu^i(x)\,d\mu^j(y)+2\sum_{i\in I}\,\int f_i\,d\mu^i,\]
where $\mathbf f=(f_i)_{i\in I}$, each $f_i:\ \mathbb R^n\to[-\infty,\infty]$ being a universally meas\-ur\-ab\-le function treated as an {\it external field\/} acting on the charges carried by the $A_i$.

The difficulties appearing in the course of our investigation are caused by the fact that a short-circuit may occur between $A_i$ and $A_j$ with $s_is_j=-1$, because these conductors may have zero Euclidean distance. See Theorem~\ref{pr1uns} below providing an example of a generalized condenser with {\it no\/} $\alpha$-Riesz energy minimizer. It is therefore meaningful to ask what kinds of additional requirements on the objects in question will prevent this blow-up effect, and secure that a solution to the corresponding $\mathbf f$-weighted minimum $\alpha$-Riesz energy problem does exist.
We show that, though the {\it closures\/} of oppositely charged plates may intersect each other even in a set of {\it nonzero\/} $\alpha$-Riesz capacity, such minimum energy problem is nevertheless solvable ({\it no\/}  short-circuit occurs) if we restrict ourselves to $\boldsymbol\mu$ with $\mu^i\leqslant\xi^i$, $i\in I$, where the constraint $\boldsymbol\xi=(\xi^i)_{i\in I}$ is properly chosen (see Sections~\ref{sec:constr} and~\ref{sec-form} for a formulation of the constrained problem). Sufficient conditions for the existence of solutions $\boldsymbol\lambda_{\mathbf A}^{\boldsymbol\xi}=(\lambda^i_{\mathbf A})_{i\in I}$ to the constrained minimum $\alpha$-Riesz energy problem are established in Theorems~\ref{th-suff} and~\ref{l:eq}; those conditions are shown in Theorem~\ref{th-unsuff} to be sharp. The uniqueness of solutions is studied in Lemma~\ref{lemma:unique:} and Corollary~\ref{cor-unique}. We also provide descriptions of the $\mathbf f$-weighted vector $\alpha$-Riesz potentials of the solutions $\boldsymbol\lambda_{\mathbf A}^{\boldsymbol\xi}$, single out their characteristic properties, and analyze the supports of the components $\lambda^i_{\mathbf A}$, $i\in I$ (Theorems~\ref{desc-pot}, \ref{desc-sup} and~\ref{zone}). The results are illustrated in Examples~\ref{ex} and~\ref{ex2}.

In particular, let $\mathbf A=(A_1,A_2)$ be a generalized condenser with the positive plate $A_1:=D$ and the negative plate $A_2:=\mathbb R^n\setminus D$, $D$ being an (open connected) bounded domain in $\mathbb R^n$ with $m_n(D)>1$ where $m_n$ is the $n$-dim\-en\-sional Lebesgue measure, and let $\mathbf f=\mathbf 0$. Then $\inf\,G_{\kappa_\alpha,\mathbf f}(\boldsymbol\mu)$ over all $\boldsymbol\mu=(\mu^1,\mu^2)$ associated with $\mathbf A$ and normalized by $\mu^i(A_i)=1$, $i=1,2$, is an {\it actual\/} minimum (although $A_2\cap C\ell_{\mathbb R^n}A_1=\partial D$) if we require additionally that $\mu^1\leqslant\xi^1:=m_n|_D$ and $\mu^2\leqslant\xi^2$, where $\xi^2$ is a positive Radon measure carried by $A_2$ and possessing the property $\xi^2\geqslant(m_n|_D)^{A_2}$ (cf.\ Theorems~\ref{th-suff} and~\ref{l:eq}). Here $m_n|_D$ denotes the restriction of $m_n$ on $D$, and $(m_n|_D)^{A_2}$ the $\alpha$-Riesz balayage of $m_n|_D$ onto~$A_2$.

The approach developed is mainly based on the simultaneous use of the vague topology and an appropriate (semi)metric structure defined in terms of the $\alpha$-Riesz energy on a set of vector measures associated with a generalized condenser (see Section~\ref{sec-Metric} for a definition of such a (semi)metric structure\footnote{A key observation behind that definition is the fact that there corresponds to every positive vector measure $\boldsymbol\mu=(\mu^i)_{i\in I}$ of finite energy associated with $\mathbf A$ a scalar ({\it signed\/}) Radon measure $R\boldsymbol\mu=\sum_{i\in I}\,s_i\mu^i$ on $\mathbb R^n$, and the mapping $R:\boldsymbol\mu\mapsto R\boldsymbol\mu$ preserves the corresponding energy semimetric (see Theorem~\ref{lemma:semimetric}). This approach extends that from \cite{ZPot1}--\cite{ZPot3} where the closures of the oppositely charged plates were assumed to be mutually disjoint.}),
 as well as on the establishment of an intimate relationship between the constrained minimum $\alpha$-Riesz energy problem and a constrained minimum $\alpha$-Green energy problem, suitably formulated. Regarding the corresponding minimum $\alpha$-Green energy problem, crucial to the arguments applied in its investigation is the {\it perfectness\/} of the $\alpha$-Green kernel $g^\alpha_D$ on a domain $D$, established recently by the authors \cite{FZ}, which amounts to the completeness in the topology defined by the energy norm $\|\nu\|_{g^\alpha_D}:=\sqrt{g^\alpha_D(\nu,\nu)}$ of the cone of all positive scalar Radon measures $\nu$ on $D$ with finite $\alpha$-Green energy $g^\alpha_D(\nu,\nu):=\iint g^\alpha_D(x,y)\,d\nu(x)\,d\nu(y)<\infty$.

\section{Preliminaries}\label{sec:princ}

Let $X$ be a locally compact (Hausdorff) space \cite[Chapter~I, Section~9, n$^\circ$\,7]{B1}, to be specified below, and $\mathfrak M(X)$ the linear
space of all real-valued scalar Radon measures $\mu$ on $X$, equipped with the {\it vague\/} topology, i.e.\ the topology of
pointwise convergence on the class $C_0(X)$ of all continuous functions\footnote{When speaking of a continuous numerical function we understand that the values are {\it finite\/} real numbers.} on $X$ with compact
support. We refer the reader to \cite{B2,Bou,E2} for the theory of measures and integration on a locally compact space, to be used throughout the paper (see also \cite{F1} for a short survey). In all that follows the integrals are understood as {\it upper\/} integrals~\cite{B2}.

For the purposes of the present study it is enough to assume that $X$ is metrizable and {\it countable at infinity\/}, where the latter means that $X$ can be represented as a countable union of compact sets \cite[Chapter~I, Section~9, n$^\circ$\,9]{B1}. Then the vague topology on $\mathfrak M(X)$ satisfies the first axiom of countability \cite[Remark~2.5]{DFHSZ2}, and the vague convergence is entirely determined by convergence of sequences. The vague topology on $\mathfrak M(X)$ is Hausdorff; hence, a vague limit of any sequence in $\mathfrak M(X)$ is {\it unique\/} (whenever it exists).

We denote by $\mu^+$ and $\mu^-$ the positive and the negative parts in the Hahn--Jordan decomposition of a measure $\mu\in\mathfrak M(X)$, and by $S^\mu_{X}=S(\mu)$ its support. A measure $\mu\in\mathfrak M(X)$ is said to be {\it bounded\/} if $|\mu|(X)<\infty$, where $|\mu|:=\mu^++\mu^-$. Let $\mathfrak M^+(X)$ stand for the (convex, vaguely closed) cone of all positive $\mu\in\mathfrak M(X)$, and let $\Psi(X)$ consist of all lower semicontinuous (l.s.c.)
functions $\psi: X\to(-\infty,\infty]$, nonnegative unless $X$ is compact.

\begin{lemma}[{\rm see e.g.\ \cite[Section~1.1]{F1}}]\label{lemma-semi}For any\/ $\psi\in\Psi(X)$ the map\/ $\mu\mapsto\langle\psi,\mu\rangle:=\int\psi\,d\mu$ is vaguely l.s.c.\ on\/~$\mathfrak M^+(X)$.\end{lemma}

We define a (function) {\it kernel\/} $\kappa(x,y)$ on $X$ as a symmetric positive function from $\Psi(X\times X)$. Given $\mu,\mu_1\in\mathfrak M(X)$, we denote by
$\kappa(\mu,\mu_1)$ and $\kappa(\cdot,\mu)$ the {\it mutual
energy\/} and the {\it potential\/} relative to the kernel $\kappa$, respectively,
i.e.\footnote{When introducing notation about numerical quantities we assume
the corresponding object on the right to be well defined~--- as a finite real number or~$\pm\infty$.}
\begin{align*}
\kappa(\mu,\mu_1)&:=\iint\kappa(x,y)\,d\mu(x)\,d\mu_1(y),\\
\kappa(x,\mu)&:=\int\kappa(x,y)\,d\mu(y),\quad x\in X.
\end{align*}
Note that $\kappa(x,\mu)$ is well defined provided that $\kappa(x,\mu^+)$ or $\kappa(x,\mu^-)$ is finite, and then $\kappa(x,\mu)=\kappa(x,\mu^+)-\kappa(x,\mu^-)$. In particular, if $\mu\in\mathfrak M^+(X)$ then $\kappa(x,\mu)$ is defined everywhere and represents a l.s.c.\ positive function on $X$ (see Lemma~\ref{lemma-semi}).
Also observe that $\kappa(\mu,\mu_1)$ is well defined and equal to $\kappa(\mu_1,\mu)$ provided that
$\kappa(\mu^+,\mu_1^+)+\kappa(\mu^-,\mu_1^-)$ or $\kappa(\mu^+,\mu_1^-)+\kappa(\mu^-,\mu_1^+)$ is finite.
For $\mu=\mu_1$, $\kappa(\mu,\mu_1)$ becomes the {\it energy\/} $\kappa(\mu,\mu)$. Let $\mathcal E_\kappa(X)$ consist
of all $\mu\in\mathfrak M(X)$ whose energy $\kappa(\mu,\mu)$ is finite, which by definition means that $\kappa(\mu^+,\mu^+)$, $\kappa(\mu^-,\mu^-)$ and $\kappa(\mu^+,\mu^-)$ are all finite, and let $\mathcal E^+_\kappa(X):=\mathcal E_\kappa(X)\cap\mathfrak M^+(X)$.

Given a set $Q\subset X$, let $\mathfrak M^+(Q;X)$ consist of all $\mu\in\mathfrak M^+(X)$ {\it concentrated on\/} (or {\it carried by\/}) $Q$, which means that $X\setminus Q$ is locally $\mu$-negligible, or equivalently that $Q$ is $\mu$-meas\-ur\-able and $\mu=\mu|_Q$, where $\mu|_Q=1_Q\cdot\mu$ is the trace (restriction) of $\mu$ on $Q$ \cite[Section~5, n$^\circ$\,2, Exemple]{Bou}. (Here $1_Q$ denotes the indicator function of $Q$.) If $Q$ is closed then $\mu$ is concentrated on $Q$ if and only if it is supported by $Q$, i.e.\ $S(\mu)\subset Q$. It follows from the countability of $X$ at infinity that the concept of local $\mu$-neg\-lig\-ibility coincides with that of $\mu$-negligibility; and hence $\mu\in\mathfrak M^+(Q;X)$ if and only if $\mu^*(X\setminus Q)=0$, $\mu^*(\cdot)$ being the {\it outer measure\/} of a set.
Write $\mathcal E_\kappa^+(Q;X):=\mathcal E_\kappa(X)\cap\mathfrak M^+(Q;X)$, $\mathfrak M^+(Q,q;X):=\{\mu\in\mathfrak M^+(Q;X):\ \mu(Q)=q\}$ and
$\mathcal E_\kappa^+(Q,q;X):=\mathcal E_\kappa(X)\cap\mathfrak M^+(Q,q;X)$, where $q\in(0,\infty)$.

Among the variety of potential-theoretic principles investigated for example in the comprehensive work by Ohtsuka~\cite{O} (see also the references therein), in the present study we shall only need the following two:
 \begin{itemize}
 \item[$\bullet$] A kernel $\kappa$ is said to satisfy the {\it complete maximum principle} (introduced by Cartan and Deny \cite{CD}) if for any $\mu\in\mathcal E^+_\kappa(X)$ and $\nu\in\mathfrak M^+(X)$ such that $\kappa(x,\mu)\leqslant\kappa(x,\nu)+c$ $\mu$-a.e., where $c\geqslant0$ is a constant, the same inequality holds everywhere on $X$.
\item[$\bullet$] A kernel $\kappa$ is said to be {\it positive definite\/} if $\kappa(\mu,\mu)\geqslant0$ for every (signed) measure $\mu\in\mathfrak M(X)$ for which the energy is well defined; and such $\kappa$ is said to be {\it strictly positive definite\/}, or to satisfy the {\it energy principle\/} if in addition $\kappa(\mu,\mu)>0$ except for $\mu=0$.
\end{itemize}

{\it Unless explicitly stated otherwise, in all that
follows we assume a kernel $\kappa$ to satisfy the energy principle\/}. Then $\mathcal E_\kappa(X)$ forms a pre-Hil\-bert space with the inner product $\kappa(\mu,\mu_1)$ and the energy norm $\|\mu\|_\kappa:=\sqrt{\kappa(\mu,\mu)}$ (see \cite{F1}). The (Hausdorff) topology
on $\mathcal E_\kappa(X)$ defined by the norm $\|\cdot\|_\kappa$ is termed {\it strong\/}.

In contrast to \cite{Fu4,Fu5} where capacity has been treated as a functional acting on positive numerical functions on $X$, in the present study we use the (standard) concept of capacity as a set function. Thus the ({\it inner\/}) {\it capacity\/} of a set $Q\subset X$ relative to the kernel $\kappa$, denoted $c_\kappa(Q)$, is defined by \begin{equation}\label{cap-def}c_\kappa(Q):=\bigl[\inf_{\mu\in\mathcal
E_\kappa^+(Q,1;X)}\,\kappa(\mu,\mu)\bigr]^{-1}\end{equation}
(see e.g.\ \cite{F1,O}). Then $0\leqslant c_\kappa(Q)\leqslant\infty$. (As usual, here and in the sequel the
infimum over the empty set is taken to be $+\infty$. We also set
$1\bigl/(+\infty)=0$ and $1\bigl/0=+\infty$.) Because of the strict positive definiteness of the kernel $\kappa$,
\begin{equation}\label{compact-fin}c_\kappa(K)<\infty\text{ \ for every compact \ }K\subset X.\end{equation}
Furthermore, by \cite[p.~153, Eq.~2]{F1},
\begin{equation}\label{compact}c_\kappa(Q)=\sup\,c_\kappa(K)\quad(K\subset Q, \ K\text{\ compact}).\end{equation}

We shall often use the fact that $c_\kappa(Q)=0$ if and only if $\mu_*(Q)=0$ for every $\mu\in\mathcal E_\kappa^+(X)$, $\mu_*(\cdot)$ being the {\it inner measure\/} of a set \cite[Lemma~2.3.1]{F1}.

As in \cite[p.\ 134]{L}, we call a measure $\mu\in\mathfrak M(X)$ {\it $c_\kappa$-absolutely continuous\/} if $\mu(K)=0$ for every compact set $K\subset X$ with $c_\kappa(K)=0$. It follows from (\ref{compact}) that for such a $\mu$, $|\mu|_*(Q)=0$ for every $Q\subset X$ with $c_\kappa(Q)=0$. Hence every $\mu\in\mathcal E_\kappa(X)$ is $c_\kappa$-ab\-sol\-utely continuous; but not conversely \cite[pp.~134--135]{L}.

\begin{definition}\label{def-perf}Following~\cite{F1}, we call a (strictly positive definite)
kernel $\kappa$ {\it perfect\/} if every strong Cauchy sequence in $\mathcal E_\kappa^+(X)$ converges strongly to any of its vague cluster points\footnote{It follows from Theorem~\ref{fu-complete} that for a perfect kernel such a vague cluster point exists and is unique.}.\end{definition}

\begin{remark}\label{rem:clas}{\rm On $X=\mathbb R^n$, $n\geqslant3$, the $\alpha$-Riesz kernel $\kappa_\alpha(x,y)=|x-y|^{\alpha-n}$, $\alpha\in(0,n)$, is strictly positive definite and perfect \cite{D1,D2}; thus so is the Newtonian kernel $\kappa_2(x,y)=|x-y|^{2-n}$ \cite{Car}. Recently it has been shown by the present authors that if $X=D$ where $D$ is an arbitrary open set in $\mathbb R^n$, $n\geqslant3$, and $g^\alpha_D$, $\alpha\in(0,2]$, is the $\alpha$-Green kernel on $D$ \cite[Chapter~IV, Section~5]{L}, then $\kappa=g^\alpha_D$ is likewise strictly positive definite and perfect \cite[Theorems~4.9 and 4.11]{FZ}.}\end{remark}

\begin{theorem}[{\rm see \cite{F1}}]\label{fu-complete} If a kernel\/ $\kappa$ on a locally compact space\/ $X$ is perfect, then the cone\/ $\mathcal E_\kappa^+(X)$ is strongly complete and the strong topology on\/ $\mathcal E_\kappa^+(X)$ is finer than the\/ {\rm(}induced\/{\rm)} vague topology on\/ $\mathcal E_\kappa^+(X)$.\end{theorem}

\begin{remark}\label{remma}{\rm In contrast to Theorem~\ref{fu-complete}, for a perfect kernel $\kappa$ the whole pre-Hilbert space $\mathcal E_\kappa(X)$ is in general strongly {\it incomplete\/}, and this is the case even for the $\alpha$-Riesz kernel of order $\alpha\in(1,n)$ on $\mathbb R^n$, $n\geqslant 3$
(see \cite{Car} and \cite[Theorem~1.19]{L}). Compare with \cite[Theorem~1]{ZUmzh} where the strong completeness has been established for the metric subspace of all ({\it signed\/}) $\nu\in\mathcal E_{\kappa_\alpha}(\mathbb R^n)$ such that $\nu^+$ and $\nu^-$ are supported by closed nonintersecting sets in $\mathbb R^n$, $n\geqslant3$. This result from \cite{ZUmzh} has been proved with the aid of Deny's theorem~\cite{D1} stating that $\mathcal E_{\kappa_\alpha}(\mathbb R^n)$ can be completed by making use of tempered distributions on $\mathbb R^n$ with finite $\alpha$-Riesz energy, defined in terms of its Fourier transform (compare with Remark~\ref{remark}).}\end{remark}

\begin{remark}\label{remark}{\rm The concept of perfect kernel is an efficient tool in minimum energy problems
over classes of {\it positive scalar\/} Radon measures with
finite energy. Indeed, if $Q\subset X$ is closed, $c_\kappa(Q)\in(0,+\infty)$, and $\kappa$ is perfect, then the minimum energy problem (\ref{cap-def}) has a unique solution $\lambda_Q$ \cite[Theorem~4.1]{F1}; we shall call such a $\lambda_Q$ the ({\it inner\/}) {\it $\kappa$-capacitary measure\/} on~$Q$. Later the concept of perfectness has been shown to be efficient also in minimum energy problems over classes of {\it vector measures\/} associated with a {\it standard condenser\/} \mbox{\cite{ZPot1}--\cite{ZPot3}} (see also Remarks~\ref{r-3} and~\ref{rem-3} below for a short survey). In contrast to \cite[Theorem~1]{ZUmzh}, the approach developed in \mbox{\cite{ZPot1}--\cite{ZPot3}} substantially used the assumption of the boundedness of the kernel on the product of the oppositely charged plates of a condenser, which made it possible to extend Cartan's proof \cite{Car} of the strong completeness of the cone $\mathcal E_{\kappa_2}^+(\mathbb R^n)$ of all {\it positive\/} measures on $\mathbb R^n$ with finite Newtonian energy to an arbitrary perfect kernel $\kappa$ on a locally compact space $X$ and suitable classes of ({\it signed\/}) measures $\mu\in\mathcal E_\kappa(X)$.}\end{remark}

\section{Minimum energy problems for a~generalized condenser in a~locally compact space}

\subsection{Vector measures associated with a generalized condenser}\label{sec:cond} A subset $L$ of a locally compact space $X$ is said to be {\it locally closed\/} if for every $x\in L$ there is a neighborhood $V$ of $x$ in $X$ such that $V\cap L$ is a closed subset of the subspace $L$ \cite[Chapter~I, Section~3, Definition~2]{B1}, or equivalently if $L$ is the intersection of an open and a closed subset of $X$ \cite[Chapter~I, Section~3, Proposition~5]{B1}.

Consider an ordered finite collection $\mathbf A=(A_i)_{i\in I}$ of nonempty, locally closed sets $A_i\subset X$ with the sign $s_i={\rm sign}\,A_i=\pm1$ prescribed. Denote $I^+:=\{i\in I: s_i=+1\}$, $I^-:=I\setminus I^+$ and $p:={\rm Card}\,I$, where $p\geqslant1$ and $I^-$ is allowed to be empty.

\begin{definition}\label{def-cond}We call $\mathbf A$ a {\it generalized condenser in\/} $X$ if $A^+\cap A^-=\varnothing$, where
\[A^+:=\bigcup_{i\in I^+}\,A_i\text{ \ and \ } A^-:=\bigcup_{j\in I^-}\,A_j.\]
\end{definition}

The sets $A_i$, $i\in I^+$, and $A_j$, $j\in I^-$, are said to be the {\it positive\/} and {\it negative\/} plates of the (generalized) condenser $\mathbf A$. To avoid trivialities, we shall always assume that
\begin{equation}\label{nonzero'}c_\kappa(A_i)>0\text{ \ for all \ }i\in I,\end{equation}
the (strictly positive definite) kernel $\kappa$ on $X$ being given. Note that any two equally signed plates may intersect each other or even coincide. Also note that, though $A_i$ and $A_j$ are disjoint for any $i\in I^+$ and $j\in I^-$, their {\it closures\/} in $X$ may intersect each other even in a set with $c_\kappa(\cdot)>0$. The concept of generalized condenser thus defined generalizes that introduced recently in \cite[Section~3]{DFHSZ2}.

\begin{definition}\label{def-cond-st}A generalized condenser $\mathbf A$ is said to be {\it standard\/} if all the (locally closed) sets $A_i$, $i\in I$, are closed in $X$.\end{definition}

Unless explicitly stated otherwise, in all that follows we assume $\mathbf A$ to be a {\it generalized\/} condenser in $X$. Let $\mathfrak M^+(\mathbf A;X)$ consist of all positive vector measures $\boldsymbol\mu=(\mu^i)_{i\in I}$ where each $\mu^i$, $i\in I$, is a positive scalar Radon measure on $X$ that is concentrated on $A_i$, i.e.
\[\mathfrak M^+(\mathbf A;X):=\prod_{i\in I}\,\mathfrak M^+(A_i;X).\]
Elements of $\mathfrak M^+(\mathbf A;X)$ are said to be ({\it vector\/}) {\it measures associated with\/} $\mathbf A$.
If a measure $\boldsymbol\mu\in\mathfrak M^+(\mathbf A;X)$ and a vector-valued function $\boldsymbol{u}=(u_i)_{i\in I}$ with $\mu^i$-meas\-ur\-able components $u_i:X\to[-\infty,\infty]$ are given, then we write
\begin{equation}\label{umu}\langle\boldsymbol{u},\boldsymbol{\mu}\rangle:=\sum_{i\in I}\,\langle u_i,\mu^i\rangle=\sum_{i\in I}\,\int u_i\,d\mu^i.\end{equation}

Being the intersection of an open and a closed subset of $X$,
each $A_i$, $i\in I$, is universally measurable, and hence $\mathfrak M^+(A_i;X)$ consists of all the restrictions $\mu|_{A_i}$ where $\mu$ ranges over $\mathfrak M^+(X)$. On the other hand, according to \cite[Chapter~I, Section~9, Proposition~13]{B1}, $A_i$ itself can be thought of as a locally compact subspace of $X$. Thus $\mathfrak M^+(A_i;X)$ consists, in fact, of all those $\nu\in\mathfrak M^+(A_i)$ for each of which there exists $\widehat{\nu}\in\mathfrak M^+(X)$ with the property
\begin{equation}\label{extend}\widehat{\nu}(\varphi)=\langle1_{A_i}\varphi,\nu\rangle\text{ \ for every \ }\varphi\in C_0(X).\end{equation}
We say that such $\widehat{\nu}$ {\it extends\/} $\nu\in\mathfrak M^+(A_i)$ by $0$ off $A_i$ to all of $X$.
A sufficient condition for (\ref{extend}) to hold is that $\nu$ be bounded.

Since $A^+\cap A^-=\varnothing$, there corresponds to each $\boldsymbol\mu\in\mathfrak M^+(\mathbf A;X)$ a ({\it signed\/}) scalar Radon measure $R_{\mathbf A}\boldsymbol\mu:=\sum_{i\in I}\,s_i\mu^i\in\mathfrak M(X)$, the 'resultant' of $\boldsymbol\mu$, whose positive and negative parts in the Hahn--Jordan decomposition are given by
\begin{equation}\label{defR}R_{\mathbf A}\boldsymbol\mu^+=\sum_{i\in I^+}\,\mu^i\text{ \ and \ }R_{\mathbf A}\boldsymbol\mu^-=\sum_{j\in I^-}\,\mu^j.\end{equation}
For the sake of brevity we shall use the short notation $R$ instead of $R_{\mathbf A}$ if this will not cause any misunderstanding.

The mapping $\mathfrak M^+(\mathbf A;X)\ni\boldsymbol\mu\mapsto R\boldsymbol\mu\in\mathfrak M(X)$ is in general non-injective. We shall call $\boldsymbol\mu,\boldsymbol\mu_1\in\mathfrak M^+(\mathbf A;X)$ {\it $R$-equivalent\/}
if $R\boldsymbol\mu=R\boldsymbol\mu_1$. Note that the relation of $R$-equ\-iv\-alence on $\mathfrak M^+(\mathbf A;X)$ is that of identity if and only if all the $A_i$, $i\in I$, are mutually disjoint. Also observe that $\boldsymbol\mu\in\mathfrak M^+(\mathbf A;X)$ is $R$-equ\-iv\-alent to $\boldsymbol 0$ (if and) only if $\boldsymbol\mu=\boldsymbol 0$.

\subsection{A (semi)metric structure on classes of vector measures}\label{sec-Metric} For a given (strictly positive definite) kernel $\kappa$ on $X$ and a given (generalized) condenser $\mathbf A$, let $\mathcal E^+_\kappa(\mathbf A;X)$ consist of all $\boldsymbol{\mu}\in\mathfrak M^+(\mathbf A;X)$ such that $\kappa(\mu^i,\mu^i)<\infty$ for all $i\in I$; in other words,
\[\mathcal E^+_\kappa(\mathbf A;X):=\prod_{i\in I}\,\mathcal E^+_\kappa(A_i;X).\]
In view of \cite[Lemma~2.3.1]{F1}, we see from (\ref{nonzero'}) that $\mathcal E^+_\kappa(\mathbf A;X)\ne\{\mathbf 0\}$.

In accordance with an electrostatic interpretation of a condenser, we say that the interaction between the components $\mu^i$, $i\in I$, of $\boldsymbol\mu\in\mathcal E^+_\kappa(\mathbf A;X)$ is characterized by the matrix $(s_is_j)_{i,j\in I}$. Given $\boldsymbol\mu,\boldsymbol\mu_1\in\mathcal E^+_\kappa(\mathbf A;X)$, we define the {\it mutual energy\/}
\begin{equation}\label{env}\kappa(\boldsymbol\mu,\boldsymbol\mu_1):=\sum_{i,j\in I}\,s_is_j\kappa(\mu^i,\mu_1^j)\end{equation}
and the {\it vector potential\/} $\kappa^{\boldsymbol\mu}=(\kappa^{\boldsymbol\mu,i})_{i\in I}$ where
\begin{equation}\label{potv}\kappa^{\boldsymbol\mu,i}(x):=\sum_{j\in I}\,s_is_j\kappa(x,\mu^j),\quad x\in X.\end{equation}

An assertion $\mathcal U(x)$ involving a variable point $x\in X$ is said to hold {\it $c_\kappa$-n.e.}\ on $Q\subset X$ if $c_\kappa(N)=0$ where $N$ consists of all $x\in Q$ for which $\mathcal U(x)$ fails to hold.

\begin{lemma}\label{l-Rpot} For any\/ $\boldsymbol\mu\in\mathcal E_\kappa^+(\mathbf A;X)$ all the\/ $\kappa^{\boldsymbol\mu,i}$, $i\in I$, are well defined and finite\/ $c_\kappa$-n.e.\ on\/ $X$. Moreover,
\begin{equation}\label{repp}\kappa^{\boldsymbol\mu,i}(\cdot)=s_i\kappa(\cdot,R\boldsymbol\mu)\text{ \ $c_\kappa$-n.e.~on \ }X.\end{equation}
\end{lemma}

\begin{proof}Since $\mu^i\in\mathcal E_\kappa^+(X)$ for every $i\in I$, $\kappa(\cdot,\mu^i)$ is finite $c_\kappa$-n.e.\ on $X$ \cite[p.~164]{F1}. Furthermore, the set of all $x\in X$ with $\kappa(x,\mu^i)=\infty$ is universally measurable, for $\kappa(\cdot,\mu^i)$ is l.s.c.\ on $X$. Combined with the fact that the inner capacity $c_\kappa(\cdot)$ is subadditive on universally measurable sets \cite[Lemma~2.3.5]{F1}, this implies that $\kappa^{\boldsymbol\mu,i}$ is well defined and finite $c_\kappa$-n.e.\ on $X$. Finally,  (\ref{repp}) is obtained directly from (\ref{defR}) and~(\ref{potv}).\end{proof}

\begin{lemma}\label{l-Ren} For any\/
$\boldsymbol\mu,\boldsymbol\mu_1\in\mathcal E^+_\kappa(\mathbf A;X)$ we have
\begin{equation}\label{Re}\kappa(\boldsymbol\mu,\boldsymbol\mu_1)=\kappa(R\boldsymbol\mu,R\boldsymbol\mu_1)\in(-\infty,\infty).\end{equation}\end{lemma}

\begin{proof} This is obtained directly from (\ref{defR}) and (\ref{env}).\end{proof}

For $\boldsymbol\mu=\boldsymbol\mu_1\in\mathcal E^+_\kappa(\mathbf A;X)$ the mutual energy $\kappa(\boldsymbol\mu,\boldsymbol\mu_1)$ becomes the {\it energy\/} $\kappa(\boldsymbol\mu,\boldsymbol\mu)$ of $\boldsymbol\mu$. Because of the strict positive definiteness of the kernel $\kappa$, we have from (\ref{Re})
\begin{equation}\label{enpos}\kappa(\boldsymbol\mu,\boldsymbol\mu)=\kappa(R\boldsymbol\mu,R\boldsymbol\mu)\in[0,\infty)\text{ \ for all \ }\boldsymbol\mu\in\mathcal E^+_\kappa(\mathbf A;X),\end{equation}
where $\kappa(\boldsymbol\mu,\boldsymbol\mu)=0$ if and only if $\boldsymbol\mu=\mathbf 0$.

In order to introduce a (semi)metric structure on $\mathcal E_\kappa^+(\mathbf A;X)$, we define
\begin{equation}\label{isom}\|\boldsymbol{\mu}-\boldsymbol{\mu}_1\|_{\mathcal E^+_\kappa(\mathbf A;X)}:=\|R\boldsymbol{\mu}-R\boldsymbol{\mu}_1\|_\kappa.\end{equation}
Based on (\ref{Re}), we see by straightforward calculation that, in fact,
\begin{equation}\label{metr-def}\|\boldsymbol{\mu}-\boldsymbol{\mu}_1\|^2_{\mathcal E^+_\kappa(\mathbf A;X)}=\sum_{i,j\in I}\,s_is_j\kappa(\mu^i-\mu_1^i,\mu^j-\mu_1^j).\end{equation}

\begin{theorem}\label{lemma:semimetric} $\mathcal E^+_\kappa(\mathbf{A};X)$ is a semimetric space with
the semimetric defined by either of the\/ {\rm(}equivalent\/{\rm)} relations\/ {\rm(\ref{isom})} or\/ {\rm(\ref{metr-def})}, and this
space is isometric to its\/ $R$-image in\/ $\mathcal E_\kappa(X)$. This semimetric is a metric if and only if all the\/ $A_i$, $i\in I$, are mutually essentially disjoint, i.e.\ with\/ $c_\kappa(A_i\cap A_j)=0$ for all\/ $i\ne j$.
\end{theorem}

\begin{proof}The former assertion is obvious by (\ref{isom}). Since a nonzero positive scalar measure of finite energy does not charge any set of zero capacity \cite[Lemma~2.3.1]{F1}, the sufficiency part of the latter assertion lemma holds. To prove the necessity part, assume on the contrary that there are two equally signed plates $A_k$ and $A_\ell$, $k\ne\ell$, with $c_\kappa(A_k\cap A_\ell)>0$. By \cite[Lemma~2.3.1]{F1}, there is a nonzero measure $\tau\in\mathcal E_\kappa^+(A_k\cap A_\ell;X)$. Choose $\boldsymbol\mu=(\mu^i)_{i\in I}\in\mathcal E^+_\kappa(\mathbf A;X)$ such that $\mu^k|_{A_k\cap A_\ell}-\tau\geqslant0$, and define $\boldsymbol\mu_m=(\mu_m^i)_{i\in I}\in\mathcal E^+_\kappa(\mathbf A;X)$, $m=1,2$, where $\mu_1^k=\mu^k-\tau$ and $\mu_1^i=\mu^i$ for all $i\ne k$, while $\mu_2^\ell=\mu^\ell+\tau$ and $\mu_2^i=\mu^i$ for all $i\ne\ell$. Then $R\boldsymbol\mu_1=R\boldsymbol\mu_2$, hence $\|\boldsymbol{\mu}_1-\boldsymbol{\mu}_2\|_{\mathcal E^+_\kappa(\mathbf A;X)}=0$, though $\boldsymbol\mu_1\ne\boldsymbol\mu_2$.\end{proof}

Similarly to the terminology for the pre-Hilbert space $\mathcal E_\kappa(X)$, the topology of the semimetric space $\mathcal E^+_\kappa(\mathbf A;X)$ is termed {\it strong\/}. We say that $\boldsymbol{\mu},\boldsymbol{\mu}_1\in\mathcal E^+_\kappa(\mathbf A;X)$ are {\it equivalent in\/} $\mathcal E^+_\kappa(\mathbf A;X)$ if $\|\boldsymbol{\mu}-\boldsymbol{\mu}_1\|_{\mathcal E^+_\kappa(\mathbf A;X)}=0$, or equivalently if $R\boldsymbol{\mu}=R\boldsymbol{\mu}_1$.

\subsection{The vague topology on $\mathfrak M^+(\mathbf A;X)$}\label{sec:vague} In Section~\ref{sec:vague} we consider a {\it standard\/} condenser $\mathbf A$ (see Definition~\ref{def-cond-st}). The set of all (vector) measures associated with $\mathbf A$ can be endowed with the {\it vague\/} topology defined as follows.

\begin{definition}\label{def-vague}
The {\it vague topology\/} on $\mathfrak M^+(\mathbf A;X)$ is the topology of the product space $\prod_{i\in I}\,\mathfrak M^+(A_i;X)$ where each of the $\mathfrak M^+(A_i;X)$ is endowed with the vague topology induced from $\mathfrak M(X)$. Namely, a sequence $\{\boldsymbol\mu_k\}_{k\in\mathbb N}\subset\mathfrak M^+(\mathbf A;X)$ converges to $\boldsymbol\mu\in\mathfrak M^+(\mathbf A;X)$ {\it vaguely\/} if for every $i\in I$, $\mu_k^i\to\mu^i$ vaguely in $\mathfrak M(X)$ as $k\to\infty$.\end{definition}

Since all the $A_i$, $i\in I$, are closed in $X$, $\mathfrak M^+(\mathbf A;X)$ is {\it vaguely closed in\/} $\mathfrak M(X)^p$. Besides, since every $\mathfrak M^+(A_i;X)$ is Hausdorff in the vague topology, so is $\mathfrak M^+(\mathbf A;X)$ \cite[Chapter~I, Section~8, Proposition~7]{B1}. Hence, a va\-gue limit of any sequence in $\mathfrak M^+(\mathbf A;X)$ {\it belongs to\/ $\mathfrak M^+(\mathbf A;X)$ and is unique\/} (whenever it exists). We call a set $\mathfrak F\subset\mathfrak M^+(\mathbf A;X)$ {\it vaguely bounded\/} if for every $\varphi\in C_0(X)$,
\[\sup_{\boldsymbol{\mu}\in\mathfrak F}\,|\mu^i(\varphi)|<\infty\text{ \ for all \ }i\in I.\]

\begin{lemma}\label{l:vague:c} A vaguely bounded set\/ $\mathfrak F\subset\mathfrak M^+(\mathbf A;X)$ is vaguely relatively
compact.\end{lemma}

\begin{proof}It is clear from the above definition that for every $i\in I$ the set
\[\mathfrak F^i:=\bigl\{\mu^i\in\mathfrak M^+(A_i;X):  \ \boldsymbol\mu=(\mu^j)_{j\in I}\in\mathfrak F\bigr\}\]
is vaguely bounded in $\mathfrak M^+(X)$; hence, by \cite[Chapitre~III, Section~2, Proposition~9]{B2}, $\mathfrak F^i$ is vaguely relatively compact in $\mathfrak M(X)$.
As $\mathfrak F\subset\prod_{i\in I}\,\mathfrak F^i$,
the lemma follows from Tychonoff's theorem on the product of compact spaces \cite[Chapter~I, Section~9, Theorem~3]{B1}.\end{proof}

\subsection{An unconstrained weighted minimum energy problem for vector measures}\label{sec:unc} Let a (generalized) condenser $\mathbf A=(A_i)_{i\in I}$ and a (strictly positive definite) kernel $\kappa$ on $X$ be given. Fix a vector-valued function $\mathbf{f}=(f_i)_{i\in I}$, where each $f_i:X\to[-\infty,\infty]$ is $\mu$-meas\-urable for every $\mu\in\mathcal E_\kappa^+(A_i;X)$ and $f_i$ is treated as an {\it external field\/} acting on the charges (measures) from $\mathcal E^+_\kappa(A_i;X)$. The {\it $\mathbf f$-weigh\-ted vector potential\/} and
the {\it $\mathbf f$-weigh\-ted energy\/} of $\boldsymbol{\mu}\in\mathcal E_\kappa^+(\mathbf A;X)$ are given by
\begin{align}\label{wpot'}\mathbf W_{\kappa,\mathbf{f}}^{\boldsymbol{\mu}}&:=\kappa^{\boldsymbol{\mu}}+\mathbf f,\\
\label{wen'}G_{\kappa,\mathbf{f}}(\boldsymbol{\mu})&:=\kappa(\boldsymbol{\mu},\boldsymbol{\mu})+2\langle\mathbf{f},\boldsymbol{\mu}\rangle,\end{align}
respectively.\footnote{$G_{\kappa,\mathbf{f}}(\mathbf{\cdot})$ is also known as the {\it Gauss functional\/} (see e.g.~\cite{O}; compare with~\cite{Gauss}). Note that when defining $G_{\kappa,\mathbf{f}}(\mathbf{\cdot})$, we have used the notation~(\ref{umu}).} Thus $\mathbf W_{\kappa,\mathbf{f}}^{\boldsymbol{\mu}}=\bigl(W_{\kappa,\mathbf{f}}^{\boldsymbol{\mu},i}\bigr)_{i\in I}$, where $W_{\kappa,\mathbf{f}}^{\boldsymbol{\mu},i}:=\kappa^{\boldsymbol\mu,i}+f_i$ (see (\ref{potv})).
Let $\mathcal E_{\kappa,\mathbf f}^+(\mathbf A;X)$ consist of all $\boldsymbol{\mu}\in\mathcal E_\kappa^+(\mathbf A;X)$ with finite $G_{\kappa,\mathbf{f}}(\boldsymbol{\mu})$, or equivalently with finite $\langle\mathbf{f},\boldsymbol{\mu}\rangle$.

\begin{lemma}\label{aux43}Suppose that a set\/ $\mathfrak E\subset\mathcal E_{\kappa,\mathbf f}^+(\mathbf A;X)$ is convex. Then there exists\/ $\boldsymbol\lambda\in\mathfrak E$ with
\begin{equation}\label{convex}G_{\kappa,\mathbf{f}}(\boldsymbol\lambda)=\min_{\boldsymbol\mu\in\mathfrak E}\,G_{\kappa,\mathbf{f}}(\boldsymbol{\mu})\end{equation}
if and only if
\begin{equation}\label{aux431}\sum_{i\in I}\,\bigl\langle W^{\boldsymbol\lambda,i}_{\kappa,\mathbf f},\mu^i-\lambda^i\bigr\rangle\geqslant0\text{ \ for all \ } \boldsymbol\mu\in\mathfrak E.\end{equation}
\end{lemma}

\begin{proof}By direct calculation, for any $\boldsymbol\mu,\boldsymbol\nu\in\mathfrak E$ and any $h\in(0,1]$ we get
\[G_{\kappa,\mathbf f}\bigl(h\boldsymbol\mu+(1-h)\boldsymbol\nu\bigr)-G_{\kappa,\mathbf f}(\boldsymbol\nu)=2h\sum_{i\in I}\,\bigl\langle W^{\boldsymbol\nu,i}_{\kappa,\mathbf f},\mu^i-\nu^i\bigr\rangle+h^2\|\boldsymbol\mu-\boldsymbol\nu\|^2_{\mathcal E^+_\kappa(\mathbf A;X)}.\]
If $\boldsymbol\nu=\boldsymbol\lambda$ satisfies (\ref{convex}), then the left (hence, also the right) side of this display is ${}\geqslant0$, which leads to (\ref{aux431}) by letting $h\to0$.
Conversely, if (\ref{aux431}) holds, then the preceding formula with $\boldsymbol\nu=\boldsymbol\lambda$ and $h=1$ gives $G_{\kappa,\mathbf f}(\boldsymbol\mu)\geqslant G_{\kappa,\mathbf f}(\boldsymbol\lambda)$  for all $\boldsymbol\mu\in\mathfrak E$, and (\ref{convex}) follows.
\end{proof}

Fix a numerical vector $\mathbf a=(a_i)_{i\in I}$ with $0<a_i<\infty$, $i\in I$, and write
\begin{align*}\mathfrak M^+(\mathbf A,\mathbf a;X)&:=\bigl\{\boldsymbol\mu\in\mathfrak M^+(\mathbf A;X): \ \mu^i(A_i)=a_i\text{ \ for all \ }i\in I\bigr\},\\
\mathcal E^+_\kappa(\mathbf A,\mathbf a;X)&:=\mathcal E^+_\kappa(\mathbf A;X)\cap\mathfrak M^+(\mathbf A,\mathbf a;X),\\
\mathcal E_{\kappa,\mathbf{f}}^+(\mathbf A,\mathbf a;X)&:=\mathcal E_{\kappa,\mathbf f}^+(\mathbf A;X)\cap\mathfrak M^+(\mathbf A,\mathbf a;X).
\end{align*}
If the class $\mathcal E_{\kappa,\mathbf{f}}^+(\mathbf A,\mathbf a;X)$ is nonempty, or equivalently if
\[G_{\kappa,\mathbf{f}}(\mathbf A,\mathbf a;X):=\inf_{\boldsymbol{\mu}\in\mathcal E_{\kappa,\mathbf{f}}^+(\mathbf A,\mathbf a;X)}\,G_{\kappa,\mathbf{f}}(\boldsymbol{\mu})<\infty,\]
then the following {\it unconstrained weighted minimum energy problem\/} makes sense.

\begin{problem}\label{prgen}Given\/ $X$, $\kappa$, $\mathbf A$, $\mathbf a$ and\/ $\mathbf f$, does there exist\/ $\boldsymbol\lambda_{\mathbf A}\in\mathcal E_{\kappa,\mathbf{f}}^+(\mathbf A,\mathbf a;X)$ with\/ \[G_{\kappa,\mathbf{f}}(\boldsymbol\lambda_{\mathbf A})=G_{\kappa,\mathbf{f}}(\mathbf A,\mathbf a;X){\rm?}\]\end{problem}

If $I=I^+=\{1\}$, $A_1$ is closed, $a_1=1$ and $f_1=0$, then Problem~\ref{prgen} reduces to the problem~(\ref{cap-def}) solved in \cite[Theorem~4.1]{F1} (see Remark~\ref{remark} above).

\begin{remark}\label{r-3}{\rm Let $\mathbf A$ be a standard condenser in $X$ and let
\begin{equation}\label{distpos}\sup_{x\in A^+, \ y\in A^-}\,\kappa(x,y)<\infty.\end{equation}
Under these assumptions, in \cite{ZPot2,ZPot3} an approach has been worked out based on both the vague and the strong topologies on $\mathcal E_\kappa^+(\mathbf A;X)$ which made it possible to provide a fairly complete analysis of Problem~\ref{prgen}. In more detail, it was shown that if the kernel $\kappa$ is perfect and if for all $i\in I$ either $f_i\in\Psi(X)$ or $f_i=s_i\kappa(\cdot,\zeta)$ for some (signed) $\zeta\in\mathcal E_\kappa(X)$, then the requirement
\begin{equation}\label{r-suff}c_\kappa(A^+\cup A^-)<\infty\end{equation}
is sufficient for Problem~\ref{prgen} to be solvable for every vector~$\mathbf a$ \cite[Theorem~8.1]{ZPot2}. However, if (\ref{r-suff}) does not hold then in general there exists a vector $\mathbf a'$ such that the problem has {\it no\/} solution (see \cite{ZPot2}).\footnote{In the case of the $\alpha$-Riesz kernels of order $1<\alpha\leqslant2$ on $\mathbb R^3$, some of the (theoretical) results on the solvability or unsolvability of Problem~\ref{prgen} mentioned in \cite{ZPot2} have been illustrated in \cite{HWZ,OWZ} by means of numerical experiments.} Therefore, it was interesting to give a description of the set of all vectors $\mathbf a$ for which Problem~\ref{prgen} is nevertheless solvable. Such a characterization has been established in~\cite{ZPot3} (see also footnote~\ref{foot} below). On the other hand,  if assumption (\ref{distpos}) is omitted, then the approach developed in \cite{ZPot2,ZPot3} breaks down, and (\ref{r-suff}) does not guarantee anymore the existence of a solution to Problem~\ref{prgen}. This has been illustrated by \cite[Theorem~4.6]{DFHSZ2} pertaining to the Newtonian kernel.}\end{remark}

\subsection{A constrained weighted minimum energy problem for vector measures}\label{sec:constr} A measure $\sigma^i\in\mathfrak M^+(A_i;X)$ is said to be a {\it constraint\/} for elements of $\mathfrak M^+(A_i,a_i;X)$ if
$\sigma^i(A_i)>a_i$. Let $\mathfrak C(A_i;X)$ consist of all these $\sigma^i$, and let
\[\mathfrak C(\mathbf A;X):=\prod_{i\in I}\,\mathfrak C(A_i;X).\]
Consider $\boldsymbol\xi=(\xi^i)_{i\in I}$ such that for each $i\in I$ either $\xi^i=\sigma^i\in\mathfrak C(A_i;X)$ or $\xi^i=\infty$, where the formal notation $\xi^i=\infty$ means that {\it no\/} upper constraint on the elements of $\mathfrak M^+(A_i,a_i;X)$ is imposed, and define
\[\mathfrak M^{\boldsymbol\xi}(\mathbf A;X):=\bigl\{\boldsymbol\mu\in\mathfrak M^+(\mathbf A;X): \ \mu^i\leqslant\xi^i\text{ \ for all \ }i\in I\bigr\}.\]
(If $\xi^i=\sigma^i\in\mathfrak C(A_i;X)$, then $\mu^i\leqslant\xi^i$ means that $\xi^i-\mu^i\geqslant0$, while we make the obvious convention that any positive scalar Radon measure is ${}\leqslant\infty$.) Also write
\begin{align*}
\mathfrak M^{\boldsymbol\xi}(\mathbf A,\mathbf a;X)&:=\mathfrak M^+(\mathbf A,\mathbf a;X)\cap\mathfrak M^{\boldsymbol\xi}(\mathbf A;X),\\
\mathcal E^{\boldsymbol\xi}_\kappa(\mathbf A,\mathbf a;X)&:=\mathcal E^+_\kappa(\mathbf A,\mathbf a;X)\cap\mathfrak M^{\boldsymbol\xi}(\mathbf A;X),\\
\mathcal E_{\kappa,\mathbf{f}}^{\boldsymbol\xi}(\mathbf A,\mathbf a;X)&:=\mathcal E_{\kappa,\mathbf{f}}^+(\mathbf A,\mathbf a;X)\cap\mathfrak M^{\boldsymbol\xi}(\mathbf A;X).\end{align*}

If the class $\mathcal E_{\kappa,\mathbf{f}}^{\boldsymbol\xi}(\mathbf A,\mathbf a;X)$ is nonempty, or equivalently if
\begin{equation*}G_{\kappa,\mathbf{f}}^{\boldsymbol\xi}(\mathbf A,\mathbf a;X):=\inf_{\boldsymbol{\mu}\in\mathcal E^{\boldsymbol\xi}_{\kappa,\mathbf{f}}(\mathbf A,\mathbf a;X)}\,G_{\kappa,\mathbf{f}}(\boldsymbol{\mu})<\infty,\end{equation*}
then the following {\it constrained weighted minimum energy problem\/} makes sense.

\begin{problem}\label{pr2gen}Given\/ $X$, $\kappa$, $\mathbf A$, $\mathbf a$, $\mathbf f$ and\/ $\boldsymbol\xi$, does there exist\/ $\boldsymbol\lambda^{\boldsymbol\xi}_{\mathbf A}\in\mathcal E^{\boldsymbol\xi}_{\kappa,\mathbf{f}}(\mathbf A,\mathbf a;X)$ such that\/ $G_{\kappa,\mathbf{f}}(\boldsymbol\lambda^{\boldsymbol\xi}_{\mathbf A})=G_{\kappa,\mathbf{f}}^{\boldsymbol\xi}(\mathbf A,\mathbf a;X)${\rm?} Let $\mathfrak S^{\boldsymbol{\xi}}_{\kappa,\mathbf f}(\mathbf A,\mathbf a;X)$ consist of all these\/~$\boldsymbol\lambda^{\boldsymbol\xi}_{\mathbf A}$.\end{problem}

\begin{lemma}\label{lemma:unique:}
Any two solutions\/ $\boldsymbol{\lambda},\breve{\boldsymbol{\lambda}}\in\mathfrak S^{\boldsymbol{\xi}}_{\kappa,\mathbf f}(\mathbf A,\mathbf a;X)$ are\/ $R$-equivalent.
\end{lemma}

\begin{proof} This can be established by standard methods based on the convexity of the class $\mathcal E^{\boldsymbol\xi}_{\kappa,\mathbf{f}}(\mathbf A,\mathbf a;X)$, the isometry between this class and its $R$-image in $\mathcal E_\kappa(X)$, and the pre-Hil\-bert structure on the space $\mathcal E_\kappa(X)$. Indeed, in view of the convexity of $\mathcal E^{\boldsymbol{\xi}}_{\kappa,\mathbf f}(\mathbf A,\mathbf a;X)$, relations (\ref{enpos}) and (\ref{wen'}) imply
\[4G_{\kappa,\mathbf f}^{\boldsymbol{\xi}}(\mathbf A,\mathbf a;X)\leqslant
4G_{\kappa,\mathbf f}\Bigl(\frac{\boldsymbol{\lambda}+\breve{\boldsymbol{\lambda}}}{2}\Bigr)=
\|R\boldsymbol{\lambda}+R\breve{\boldsymbol{\lambda}}\|_\kappa^2+
4\langle\mathbf{f},\boldsymbol{\lambda}+\breve{\boldsymbol{\lambda}}\rangle.\]
On the other hand, applying the parallelogram identity in $\mathcal E_\kappa(X)$ to $R\boldsymbol{\lambda}$
and $R\breve{\boldsymbol{\lambda}}$ and then adding and
subtracting
$4\langle\mathbf{f},\boldsymbol{\lambda}+\breve{\boldsymbol{\lambda}}\rangle$ we get
\[\|R\boldsymbol{\lambda}-R\breve{\boldsymbol{\lambda}}\|_\kappa^2=
-\|R\boldsymbol{\lambda}+R\breve{\boldsymbol{\lambda}}\|_\kappa^2-4\langle\mathbf{f},\boldsymbol{\lambda}+
\breve{\boldsymbol{\lambda}}\rangle+2G_{\kappa,\mathbf{f}}(\boldsymbol{\lambda})+
2G_{\kappa,\mathbf{f}}(\breve{\boldsymbol{\lambda}}).\] When combined
with the preceding relation, this yields
\[0\leqslant\|R\boldsymbol{\lambda}-R\breve{\boldsymbol{\lambda}}\|^2_\kappa\leqslant-
4G_{\kappa,\mathbf{f}}^{\boldsymbol{\xi}}(\mathbf{A},\mathbf{a};X)+2G_{\kappa,\mathbf{f}}(\boldsymbol{\lambda})+
2G_{\kappa,\mathbf{f}}(\breve{\boldsymbol{\lambda}})=0,\] which
establishes the lemma because of the strict positive definiteness of~$\kappa$.\end{proof}

\begin{corollary}\label{cor-unique} If the class\/ $\mathfrak S^{\boldsymbol{\xi}}_{\kappa,\mathbf f}(\mathbf A,\mathbf a;X)$ is nonempty, then it reduces to a single element whenever all the\/ $A_i$, $i\in I$, are mutually essentially disjoint.\end{corollary}

\begin{proof}This follows directly from Lemma~\ref{lemma:unique:} and Theorem~\ref{lemma:semimetric}.\end{proof}

\begin{remark}\label{rem-3}{\rm Assume for a moment that (\ref{distpos}) holds, the condenser $\mathbf A$ is standard, the kernel $\kappa$ is perfect and the external field $\mathbf f$ is as described in Remark~\ref{r-3}. It has been shown in \cite[Theorem~6.2]{Z9} that then condition (\ref{r-suff}) guarantees the existence of a solution to Problem~\ref{pr2gen} for any $\boldsymbol\xi$ and any vector $\mathbf a$.\footnote{Actually, this result and those described in Remark~\ref{r-3} have been obtained in \mbox{\cite{Z9}--\cite{ZPot3}} even for {\it infinite\/} dimensional vector measures.\label{foot}}}\end{remark}

\section{$\alpha$-Riesz balayage and $\alpha$-Green kernel}\label{sec:RG} In all that follows fix $n\geqslant3$, $\alpha\in(0,2]$ and a domain $D\subset\mathbb R^n$ with $c_{\kappa_\alpha}(D^c)>0$, where $D^c:=\mathbb R^n\setminus D$, and assume that either $\kappa(x,y)=\kappa_\alpha(x,y):=|x-y|^{\alpha-n}$ is the {\it $\alpha$-Riesz kernel\/} on $X=\mathbb R^n$, or $\kappa(x,y)=g_D^\alpha(x,y)$ is the {\it $\alpha$-Green kernel\/} on $X=D$. For the definition of $g_D^\alpha$, see \cite[Chapter~IV, Section~5]{L} or see below.

For given $x\in\mathbb R^n$ and $r\in(0,\infty)$ write $B(x,r):=\{y\in\mathbb R^n:\ |y-x|<r\}$, $S(x,r):=\{y\in\mathbb R^n:\ |y-x|=r\}$ and $\overline{B}(x,r):=B(x,r)\cup S(x,r)$.
Let $\partial Q$ denote the boundary of a set $Q\subset\mathbb R^n$ in the topology of~$\mathbb R^n$.

We shall simply write $\alpha$ instead of $\kappa_\alpha$ if $\kappa_\alpha$ serves as an index.
When speaking of a positive scalar Radon measure $\mu\in\mathfrak M^+(\mathbb R^n)$, we always assume $\kappa_\alpha(\cdot,\mu)\not\equiv+\infty$.
This implies that
\begin{equation}\label{1.3.10}\int_{|y|>1}\,\frac{d\mu(y)}{|y|^{n-\alpha}}<\infty\end{equation}
(see \cite[Eq.~1.3.10]{L}), and consequently that $\kappa_\alpha(\cdot,\mu)$ is finite $c_\alpha$-n.e.\ on $\mathbb R^n$ \cite[Chap\-ter~III, Section~1]{L}; these two implications can actually be reversed.

We shall usually use the short form 'n.e.' instead of '$c_\alpha$-n.e.' if this will not cause any mis\-under\-standing.

\begin{definition}\label{d-ext} $\nu\in\mathfrak M(D)$ is called {\it extendible\/} if there exist $\widehat{\nu^+}$ and $\widehat{\nu^-}$ extending $\nu^+$ and $\nu^-$, respectively, by~$0$ off $D$ to $\mathbb R^n$ (see (\ref{extend})), and if these $\widehat{\nu^+}$ and $\widehat{\nu^-}$ satisfy (\ref{1.3.10}).
We identify such a $\nu\in\mathfrak M(D)$ with its extension $\widehat\nu:=\widehat{\nu^+}-\widehat{\nu^-}$, and we therefore write $\widehat\nu=\nu$.\end{definition}

Every bounded measure $\nu\in\mathfrak M(D)$ is extendible. The converse holds if $D$ is bounded, but not in general (e.g.\ not if $D^c$ is compact). The set of all extendible measures consists of all the restrictions $\mu|_D$ where $\mu$ ranges over $\mathfrak M(\mathbb R^n)$.

The {\it $\alpha$-Green kernel\/} $g=g_D^\alpha$ on $D$ is defined by
\begin{equation*}g^\alpha_D(x,y)=\kappa_\alpha(x,\varepsilon_y)-
\kappa_\alpha(x,\varepsilon_y^{D^c})\text{ \ for all \ }x,y\in D,\end{equation*}
where $\varepsilon_y$ denotes the unit Dirac measure at a
point $y$ and $\varepsilon_y^{D^c}$ its {\it $\alpha$-Riesz balayage\/} onto the (closed) set $D^c$, uniquely determined in the frame of the classical approach by \cite[Theorem~3.6]{FZ} pertaining to positive Radon measures on $\mathbb R^n$. See also the book by Bliedtner and Hansen \cite{BH} where balayage is studied in the setting of balayage spaces.

We shall simply write $\mu'$ instead of $\mu^{D^c}$ when speaking of the $\alpha$-Riesz balayage of $\mu\in\mathfrak M^+(D;\mathbb R^n)$ onto $D^c$. According to \cite[Corollaries~3.19 and 3.20]{FZ}, for any $\mu\in\mathfrak M^+(D;\mathbb R^n)$ {\it the balayage\/ $\mu'$ is\/ $c_\alpha$-absolutely continuous and it is determined uniquely by relation
\begin{equation}\label{bal-eq}\kappa_\alpha(\cdot,\mu')=\kappa_\alpha(\cdot,\mu)\text{ \ n.e.\ on \ }D^c\end{equation}
among the\/ $c_\alpha$-absolutely continuous measures supported by\/} $D^c$. Furthermore, there holds the integral representation\footnote{In the literature the integral representation (\ref{int-repr}) seems to have been more or less taken for granted, though it has been pointed out in \cite[p.~18, Remarque]{Bou} that it requires that the family $(\varepsilon_y')_{y\in D}$ is {\it $\mu$-ad\-equ\-ate\/} in the sense of \cite[Section~3, D\'efinition~1]{Bou} (see also counterexamples (without $\mu$-adequacy) in Exercises~1 and~2 at the end of that section). We therefore have brought in \cite[Lemma~3.16]{FZ} a proof of this adequacy.}
\begin{equation}\label{int-repr}\mu'=\int\varepsilon_y'\,d\mu(y)\end{equation}
(see \cite[Theorem~3.17]{FZ}). If moreover\/ $\mu\in\mathcal E_\alpha^+(D;\mathbb R^n)$, then the balayage $\mu'$ is in fact
{\it the orthogonal projection of\/ $\mu$ onto the convex cone\/ $\mathcal E^+_\alpha(D^c;\mathbb R^n)$} (see \cite[Theorem~4.12]{Fu5} or \cite[Theorem~3.1]{FZ}), i.e.\ $\mu'\in\mathcal E^+_\alpha(D^c;\mathbb R^n)$ and
\begin{equation}\label{proj}\|\mu-\theta\|_\alpha>\|\mu-\mu'\|_\alpha\text{ \ for all \ }\theta\in\mathcal E^+_\alpha(D^c;\mathbb R^n),\quad\theta\ne\mu'.\end{equation}

If now $\nu\in\mathfrak M(D)$ is an extendible (signed) measure, then
\[\nu':=\nu^{D^c}:=(\nu^+)^{D^c}-(\nu^-)^{D^c}\]
is said to be a {\it balayage\/} of $\nu$ onto $D^c$. It follows from \cite[Chapter~III, Section~1, n$^\circ$\,1, Remark]{L} that the balayage $\nu'$ is determined uniquely by (\ref{bal-eq}) with $\nu$ in place of $\mu$ among the $c_\alpha$-absolutely continuous measures supported by~$D^c$.

The following definition goes back to Brelot (see \cite[Theorem~VII.13]{Brelo2}).

\begin{definition}A closed set $F\subset\mathbb R^n$ is said to be {\it $\alpha$-thin
at infinity\/} if either $F$ is compact, or the inverse of~$F$ relative to $S(0,1)$ has $x=0$ as an $\alpha$-irregular boundary point (cf.\ \cite[Theorem~5.10]{L}).\end{definition}

\begin{theorem}[{\rm see \cite[Theorem~3.22]{FZ}}]\label{bal-mass-th} The set\/ $D^c$ is not\/ $\alpha$-thin at infinity if and only if for every bounded measure\/ $\mu\in\mathfrak M^+(D)$ we have
\begin{equation}\label{t-mass'}\mu'(\mathbb R^n)=\mu(\mathbb R^n).\end{equation}
\end{theorem}

As noted in Remark~\ref{rem:clas} above, the $\alpha$-Riesz kernel $\kappa_\alpha$ on $\mathbb R^n$ and the $\alpha$-Green kernel $g^\alpha_D$ on $D$ are both strictly positive definite and moreover perfect. Furthermore, the $\alpha$-Riesz kernel $\kappa_\alpha$ (with $\alpha\in(0,2]$) satisfies the complete maximum principle in the form stated in Section~\ref{sec:princ} (see \cite[Theorems~1.27, 1.29]{L}). Regarding a similar result for the $\alpha$-Green kernel $g$, the following assertion holds.

\begin{theorem}[{\rm see \cite[Theorem~4.6]{FZ}}]\label{th-dom-pr} Let\/ $\mu,\nu\in\mathfrak M^+(D)$ be extendible, $g(\mu,\mu)<\infty$, and let\/ $w$ be a positive\/ $\alpha$-super\-har\-monic function on\/ $\mathbb R^n$ {\rm\cite[{\it Chapter\/}~I, {\it Section\/}~5, {\it n\/}$^\circ$\,20]{L}}. If moreover\/ $g(\cdot,\mu)\leqslant g(\cdot,\nu)+w(\cdot)$ $\mu$-a.e.\ on\/ $D$, then the same inequality holds on all of\/~$D$.\end{theorem}

The following three lemmas establish relations between potentials and energies relative to the kernels $\kappa_\alpha$ and $g=g^\alpha_D$, respectively.

\begin{lemma}\label{l-hatg} For any extendible measure\/ $\mu\in\mathfrak M(D)$, the $\alpha$-Green potential\/ $g(\cdot,\mu)$ is well defined and finite\/ {\rm(}$c_\alpha$-{\rm)}n.e.\ on\/ $D$ and given by\footnote{If $Q$ is a given subset of $D$, then any assertion involving a variable point holds n.e.\ on $Q$ if and only if it holds $c_g$-n.e.\ on $Q$ \cite[Lemma~2.6]{DFHSZ}.\label{RG}}
\begin{equation}\label{hatg}g(\cdot,\mu)=\kappa_\alpha(\cdot,\mu-\mu')\text{ \ n.e.\ on \ }D.\end{equation}
\end{lemma}

\begin{proof}It is seen from Definition~\ref{d-ext} that $\kappa_\alpha(\cdot,\mu)$ is well defined and finite n.e.\ on $\mathbb R^n$, and hence so is $\kappa_\alpha(\cdot,\mu')$. Applying (\ref{int-repr}) to $\mu^\pm$, we get by \cite[Section~3, Th\'eor\`eme~1]{Bou}
 \[g(\cdot,\mu)=\int\,\bigl[\kappa_\alpha(\cdot,\varepsilon_y)-\kappa_\alpha(\cdot,\varepsilon_y')\bigr]\,d\mu(y)=\kappa_\alpha(\cdot,\mu)-\kappa_\alpha(\cdot,\mu')\]
n.e.\ on $D$, and the lemma follows.\end{proof}

\begin{lemma}\label{l-hen'} Suppose that\/ $\mu\in\mathfrak M(D)$ is extendible and the extension belongs to\/ $\mathcal E_\alpha(\mathbb R^n)$. Then
\begin{align}
\label{l3-1}&\mu\in\mathcal E_g(D),\\
\label{l3-2}&\mu-\mu'\in\mathcal E_\alpha(\mathbb R^n),\\
\label{eq1-2-hen}&\|\mu\|^2_g=\|\mu-\mu'\|^2_\alpha=\|\mu\|^2_\alpha-\|\mu'\|^2_\alpha.
\end{align}
\end{lemma}

\begin{proof}In view of the definition of a (signed) measure of finite energy (see Section~\ref{sec:princ}), we obtain (\ref{l3-1}) from the inequality\footnote{The strict inequality in (\ref{g-ineq}) is caused by our convention that $c_\alpha(D^c)>0$.}
\begin{equation}\label{g-ineq}g^\alpha_D(x,y)<\kappa_\alpha(x,y)\text{ \ for all \ }x,y\in D,\end{equation}
while (\ref{l3-2}) from \cite[Corollary~3.7]{FZ} (or \cite[Theorems~3.1 and~3.6]{FZ}).
According to Lemma~\ref{l-hatg} and footnote~\ref{RG}, $g(\cdot,\mu)$ is finite $c_g$-n.e.\ on $D$ and given by (\ref{hatg}), while by (\ref{l3-1}) the same holds $|\mu|$-a.e.\ on $D$ (see \cite[Lemma~2.3.1]{F1}). Integrating (\ref{hatg}) with respect to $\mu^\pm$, we therefore obtain by subtraction
\begin{equation}\label{l1}\infty>g(\mu,\mu)=\kappa_\alpha(\mu-\mu',\mu).\end{equation}
Since $\kappa_\alpha(\cdot,\mu-\mu')=0$ n.e.\ on $D^c$ by (\ref{bal-eq}) and since $\mu'$ is $c_\alpha$-absolutely continuous, we also have
\begin{equation}\label{l2}\kappa_\alpha(\mu-\mu',\mu')=0,\end{equation}
which results in the former equality in (\ref{eq1-2-hen}) when combined with (\ref{l1}). Due to (\ref{l3-2}), relation (\ref{l2}) takes the form $\|\mu'\|^2_\alpha=\kappa_\alpha(\mu,\mu')$,
and the former equality in (\ref{eq1-2-hen}) implies the latter.\end{proof}

\begin{lemma}\label{l-hen'-comp}
Assume that\/ $\mu\in\mathfrak M(D)$ has compact support\/ $S^\mu_D$. Then\/ $\mu\in\mathcal E_g(D)$ if and only if its extension belongs to\/ $\mathcal E_\alpha(\mathbb R^n)$.\end{lemma}

\begin{proof}According to Lemma~\ref{l-hen'}, it is enough to establish the 'only if' part of the lemma. We may clearly assume that $\mu$ is positive. Since $\kappa_\alpha(\cdot,\mu')$ is continuous on $D$ and hence bounded on the compact set $S_D^{\mu}$, we have
\begin{equation}\label{comp-f-e}\kappa_\alpha(\mu,\mu')<\infty.\end{equation} But $g(\mu,\mu)$ is finite by assumption, and therefore likewise as in the preceding proof relation (\ref{l1}) holds. Combining (\ref{l1}) with (\ref{comp-f-e}) yields $\mu\in\mathcal E_\alpha(\mathbb R^n)$.
\end{proof}

\section{Minimum $\alpha$-Riesz energy problems for generalized condensers}\label{sec:pr1}
\subsection{An unconstrained weighted minimum $\alpha$-Riesz energy problem}\label{sec-def}
Consider a generalized condenser $\mathbf A=(A_i)_{i\in I}$ in $\mathbb R^n$ with $p:={\rm Card}\,I\geqslant 2$ such that $I^+:=\{1,\ldots,p-1\}$ and $I^-:=\{p\}$ (see Section~\ref{sec:cond}). Also require that the negative plate $A_p$ is {\it closed\/} in $\mathbb R^n$, while all the positive plates $A_j$, $j\in I^+$, are {\it relatively closed\/} subsets of the (open) set $D:=A_p^c=\mathbb R^n\setminus A_p$.\footnote{By \cite[Chapter~I, Section~3, Proposition~5]{B1}, this is in agreement with our general requirement that the sets $A_i$, $i\in I$ be locally closed in $\mathbb R^n$ (see the beginning of Section~\ref{sec:cond}).}  For the sake of simplicity, in all that follows assume that $D$ is a domain.
Recall that, by convention (\ref{nonzero'}), $c_\alpha(A_i)>0$ for all $i\in I$.

When speaking of an external field $\mathbf{f}=(f_i)_{i\in I}$ acting on the vector measures of the class $\mathcal E^+_\alpha(\mathbf A;\mathbb R^n)$, we shall always tacitly assume that
either Case~I or Case~II holds, where
\begin{itemize}
\item[\rm I.] {\it $f_i\in\Psi(\mathbb R^n)$ for every\/ $i\in I$ and  moreover}
\begin{equation}\label{fp}f_p=0\quad\text{\it n.e.\ on \ } A_p.\end{equation}
\item[\rm II.] {\it $f_i=s_i\kappa_{\alpha}(\cdot,\zeta-\zeta')$ for every\/ $i\in I$, where\/ $\zeta$ is an extendible\/ {\rm(}signed\/{\rm)} Radon measure on\/ $D$ with\/ $\kappa_\alpha(\zeta,\zeta)<\infty$\/}.
\end{itemize}

Observe that (\ref{fp}) holds also in Case II (see (\ref{bal-eq})). Since a nonzero positive scalar measure of finite energy does not charge any set of zero capacity \cite[Lemma~2.3.1]{F1}, we thus see that under the stated assumptions no external field acts on the measures from $\mathcal E^+_\alpha(A_p;\mathbb R^n)$.
Furthermore, $D^c$ is $\nu$-negligible for any $\nu\in\mathfrak M^+(A_j;\mathbb R^n)$, $j\in I^+$ (see Section~\ref{sec:princ}). We are thus led to the following conclusion.

\begin{lemma}\label{conclusion}For any\/ $\boldsymbol{\mu}\in\mathcal E_\alpha^+(\mathbf A;\mathbb R^n)$, $G_{\alpha,\mathbf{f}}(\boldsymbol{\mu})$ can\/ {\it(}equivalently\/{\rm)} be defined as\footnote{Cf.\ (\ref{wen'}) with $X=\mathbb R^n$ and $\kappa=\kappa_\alpha$.}
\begin{equation}\label{plus}G_{\alpha,\mathbf{f}}(\boldsymbol{\mu})=\kappa_\alpha(\boldsymbol{\mu},\boldsymbol{\mu})+2\langle\mathbf{f},\boldsymbol{\mu}\rangle=
\kappa_\alpha(\boldsymbol{\mu},\boldsymbol{\mu})+2\langle\mathbf{f}^+,\boldsymbol{\mu}^+\rangle,\end{equation}
where\/ ${\mathbf f}^+:=(f_j|_D)_{j\in I^+}$ and\/ ${\boldsymbol\mu}^+:=(\mu^j)_{j\in I^+}$.\end{lemma}

If Case II holds, then for every $\boldsymbol{\mu}\in\mathcal E_\alpha^+(\mathbf A;\mathbb R^n)$ we get from (\ref{defR}) and (\ref{enpos})
\[G_{\alpha,\mathbf{f}}(\boldsymbol{\mu})=\|R\boldsymbol{\mu}\|^2_\alpha+2\sum_{i\in I}\,s_i\kappa_\alpha(\zeta-\zeta',\mu^i)=\|R\boldsymbol{\mu}\|^2_\alpha+2\kappa_\alpha(\zeta-\zeta',R\boldsymbol{\mu}),\]
hence
\begin{equation}\label{GCII}\infty>G_{\alpha,\mathbf{f}}(\boldsymbol{\mu})=
\|R\boldsymbol{\mu}+\zeta-\zeta'\|_\alpha^2-\|\zeta-\zeta'\|_\alpha^2\geqslant-\|\zeta-\zeta'\|_\alpha^2>-\infty.\end{equation}

Thus in either Case I or Case II
\begin{equation*}\label{Glowerb}G_{\alpha,\mathbf f}(\boldsymbol\mu)\geqslant-M>-\infty\text{ \ for all \ }\boldsymbol{\mu}\in\mathcal E_\alpha^+(\mathbf A;\mathbb R^n),\end{equation*}
which is clear from (\ref{enpos}) and (\ref{plus}) if Case I holds, or from (\ref{GCII}) otherwise.

Fix a numerical vector $\mathbf a=(a_i)_{i\in I}$ with $a_i>0$, $i\in I$. Using the notations of Section~\ref{sec:unc} with $X=\mathbb R^n$ and $\kappa=\kappa_\alpha$, we obtain from the preceding display
\begin{equation}\label{est-bel}G_{\alpha,\mathbf f}(\mathbf A,\mathbf a;\mathbb R^n):=\inf_{\boldsymbol{\mu}\in\mathcal E_{\alpha,\mathbf{f}}^+(\mathbf A,\mathbf a;\mathbb R^n)}\,G_{\alpha,\mathbf f}(\boldsymbol{\mu})>-\infty.\end{equation}
If $\mathcal E_{\alpha,\mathbf{f}}^+(\mathbf A,\mathbf a;\mathbb R^n)$ is nonempty, or equivalently if $G_{\alpha,\mathbf f}(\mathbf A,\mathbf a;\mathbb R^n)<\infty$,
then we can consider (the unconstrained) Problem~\ref{prgen} on the existence of $\boldsymbol\lambda_{\mathbf A}\in\mathcal E_{\alpha,\mathbf{f}}^+(\mathbf A,\mathbf a;\mathbb R^n)$ with
\[G_{\alpha,\mathbf f}(\boldsymbol\lambda_{\mathbf A})=G_{\alpha,\mathbf f}(\mathbf A,\mathbf a;\mathbb R^n).\]
The following theorem shows that, in general, this problem has {\it no\/} solution.

\begin{theorem}\label{pr1uns}Suppose that\/ $D^c$ is not\/ $\alpha$-thin at infinity, $I^+=\{1\}$, $c_{g^\alpha_D}(A_1)=\infty$, and let\/ $\mathbf a=\mathbf 1$, $\mathbf f=\mathbf 0$. Then \[G_{\alpha,\mathbf f}(\mathbf A,\mathbf a;\mathbb R^n)=\bigl[c_{g^\alpha_D}(A_1)\bigr]^{-1}=0;\]
hence\/ $G_{\alpha,\mathbf f}(\mathbf A,\mathbf a;\mathbb R^n)$ cannot be an actual minimum because\/ $\boldsymbol{0}\notin\mathcal E^+_{\alpha,\mathbf f}(\mathbf A,\mathbf a;\mathbb R^n)$.\end{theorem}

\begin{proof} Since $G_{\alpha,\mathbf f}(\boldsymbol\mu)=\kappa_\alpha(\boldsymbol\mu,\boldsymbol\mu)$ because of $\mathbf f=\mathbf 0$, Problem~\ref{prgen} reduces to the problem of minimizing $\kappa_\alpha(\boldsymbol\mu,\boldsymbol\mu)$ over $\mathcal E_\alpha^+(\mathbf A,\mathbf a;\mathbb R^n)$. Thus by (\ref{enpos})
\begin{equation}\label{geq}G_{\alpha,\mathbf f}(\mathbf A,\mathbf a;\mathbb R^n)\geqslant0.\end{equation}
Consider compact sets $K_\ell\subset A_1$, $\ell\in\mathbb N$, such that $K_\ell\uparrow A_1$ as $\ell\to\infty$. By (\ref{compact}),
\begin{equation}\label{incr}c_g(K_\ell)\uparrow c_g(A_1)=\infty\text{ \ as \ }\ell\to\infty,\end{equation}
and hence there is no loss of generality in assuming that $c_g(K_\ell)>0$ for every $\ell\in\mathbb N$. Furthermore, since the $\alpha$-Green kernel $g$ is strictly positive definite and moreover perfect (Remark~\ref{rem:clas}), we see from (\ref{compact-fin}) that $c_g(K_\ell)<\infty$ and, by Remark~\ref{remark}, there exists a (unique) $g$-capacitary measure $\lambda_\ell$ on $K_\ell$, i.e.\ $\lambda_\ell\in\mathcal E_g^+(K_\ell,1;D)$ with
\[\|\lambda_\ell\|^2_g=1/c_g(K_\ell)<\infty.\]
According to Lemma~\ref{l-hen'-comp} with $\lambda_\ell$ in place of $\mu$, $\kappa_\alpha(\lambda_\ell,\lambda_\ell)$ is finite along with $g(\lambda_\ell,\lambda_\ell)$. Hence, by Lemma~\ref{l-hen'},
\[\|\lambda_\ell\|^2_g=\|\lambda_\ell-\lambda_\ell'\|^2_\alpha.\]
Applying Theorem~\ref{bal-mass-th}, we get $\boldsymbol\mu_\ell:=(\lambda_\ell,\lambda_\ell')\in\mathcal E_{\alpha,\mathbf{f}}^+(\mathbf A,\mathbf a;\mathbb R^n)$, which together with the two preceding displays and  (\ref{enpos}) and (\ref{geq}) gives
\[1/c_g(K_\ell)=\|\lambda_\ell-\lambda_\ell'\|^2_\alpha=\kappa_\alpha(\boldsymbol\mu_\ell,\boldsymbol\mu_\ell)\geqslant G_{\alpha,\mathbf f}(\mathbf A,\mathbf a;\mathbb R^n)\geqslant0.\]
Letting here $\ell\to\infty$, we obtain the theorem from (\ref{incr}).
\end{proof}

Using the electrostatic interpretation, which is possible for the Coulomb kernel $|x-y|^{-1}$ on $\mathbb R^3$, we say that under the hypotheses of Theorem~\ref{pr1uns} a {\it short-circuit occurs\/} between the oppositely signed plates of the generalized condenser $\mathbf A$. It is therefore meaningful to ask what kinds of additional requirements on the objects in question will prevent this blow-up effect, and secure that a solution to the corresponding minimum $\alpha$-Riesz energy problem does exist. To this end we have succeeded in working out a substantive theory by imposing a proper upper constraint on the vector measures under consideration.

\subsection{A constrained weighted minimum $\alpha$-Riesz energy problem}\label{sec-form}
Let $\mathbf A$, $\mathbf a$ and $\mathbf f$ be as at the beginning of Section~\ref{sec-def}. {\it In the rest of the paper we assume additionally that\/ $A_p$ $({}=D^c)$ is not\/ $\alpha$-thin at infinity and}
\begin{equation}\label{ap}a_p=\sum_{j\in I^+}\,a_j.\end{equation}
Using notation of Section~\ref{sec:constr}, fix $\boldsymbol\xi=(\xi^i)_{i\in I}$ with
\begin{equation}\label{constr1}\xi^j\in\mathfrak C(A_j;\mathbb R^n)\cap\mathcal E^+_\alpha(A_j;\mathbb R^n)\text{ \ for all \ }j\in I^+,\text{ \ and \ }\xi^p=\infty.\end{equation}
Unless explicitly stated otherwise, for these $\mathbf A$, $\mathbf a$, $\mathbf f$, and $\boldsymbol\xi$ we shall always require that\footnote{  $G_{\alpha,\mathbf{f}}^{\boldsymbol\xi}(\mathbf A,\mathbf a;\mathbb R^n)$ is then actually finite, for $G_{\alpha,\mathbf{f}}^{\boldsymbol\xi}(\mathbf A,\mathbf a;\mathbb R^n)>-\infty$ by $\mathcal E_{\alpha,\mathbf{f}}^{\boldsymbol\xi}(\mathbf A,\mathbf a;\mathbb R^n)\subset\mathcal E_{\alpha,\mathbf{f}}^+(\mathbf A,\mathbf a;\mathbb R^n)$ and~(\ref{est-bel}).\label{foot-G-finite}}
\begin{equation}\label{cggen}G_{\alpha,\mathbf{f}}^{\boldsymbol\xi}(\mathbf A,\mathbf a;\mathbb R^n)<\infty.\end{equation}

{\it The main purpose of this paper is to analyze Problem\/~{\rm\ref{pr2gen}} on the existence of\/ $\boldsymbol\lambda_{\mathbf A}^{\boldsymbol\xi}\in\mathcal E_{\alpha,\mathbf{f}}^{\boldsymbol\xi}(\mathbf A,\mathbf a;\mathbb R^n)$ with\/} $G_{\alpha,\mathbf f}(\boldsymbol\lambda_{\mathbf A}^{\boldsymbol\xi})=G_{\alpha,\mathbf{f}}^{\boldsymbol\xi}(\mathbf A,\mathbf a;\mathbb R^n)$. Recall that $\mathfrak S^{\boldsymbol{\xi}}_{\alpha,\mathbf f}(\mathbf A,\mathbf a;\mathbb R^n)$ denotes the class of all solutions to this problem (provided these exist). By Lemma~\ref{lemma:unique:}, any two solutions are $R$-equivalent, while by  Corollary~\ref{cor-unique}, $\mathfrak S^{\boldsymbol{\xi}}_{\alpha,\mathbf f}(\mathbf A,\mathbf a;\mathbb R^n)$ reduces to a single element whenever the $A_j$, $j\in I^+$, are mutually essentially disjoint.

Conditions on $\mathbf A$, $\mathbf f$ and $\boldsymbol\xi$ which guarantee that (\ref{cggen}) holds are given in the following Lemma~\ref{suff-fin}. Write
\begin{equation}\label{Acirc}A_j^\circ:=\bigl\{x\in A_j: \ |f_j(x)|<\infty\bigr\},\quad j\in I^+.\end{equation}

\begin{lemma}\label{suff-fin}
Relation\/ {\rm(\ref{cggen})} holds if either Case\/~{\rm II} takes place, or\/ {\rm(}in the presence of Case\/~{\rm I}{\rm)} if
\begin{equation}\label{acirc}\xi^j(A_j\setminus A_j^\circ)=0\text{ \ for all \ }j\in I^+.\end{equation}
\end{lemma}

\begin{proof}Assume first that (\ref{acirc}) holds. Then for every $j\in I^+$ we have $\xi^j(A_j^\circ)>a_j$, and by the universal measurability of $A_j^\circ$ there is  a compact set $K_j\subset A_j^\circ$ such that $\xi^j(K_j)>a_j$ and $|f_j|\leqslant M_j<\infty$ on $K_j$ for some constant $M_j$ (see (\ref{Acirc})). Define $\boldsymbol\mu:=(\mu^i)_{i\in I}$, where $\mu^j:=a_j\xi^j|_{K_j}/\xi^j(K_j)$ for all $j\in I^+$ and $\mu^p$ is any measure from $\mathcal E^+_\alpha(A_p,a_p;\mathbb R^n)$ (such $\mu^p$ exists since $c_\alpha(A_p)>0$). Noting that $\xi^j|_{K_j}\in\mathcal E_\alpha^+(K_j;\mathbb R^n)$ for all $j\in I^+$ by (\ref{constr1}), we get $\boldsymbol\mu\in\mathcal E_{\alpha,\mathbf{f}}^{\boldsymbol\xi}(\mathbf A,\mathbf a;\mathbb R^n)$ which yields (\ref{cggen}).
To complete the proof, it is left to observe that (\ref{acirc}) holds automatically if Case~II takes place, because then $\kappa_\alpha(\cdot,\zeta-\zeta')$ is finite n.e.\ on $\mathbb R^n$, hence $\xi^j$-a.e.\ for all $j\in I^+$ by \cite[Lemma~2.3.1]{F1}.\end{proof}

The theory developed in the present study includes sufficient and/or necessary conditions for the existence of solutions $\boldsymbol\lambda_{\mathbf A}^{\boldsymbol\xi}=(\lambda_{\mathbf A}^i)_{i\in I}$ to Problem~\ref{pr2gen} with $\mathbf A$, $\mathbf a$, $\mathbf f$ and $\boldsymbol\xi$ chosen above (see Theorems~\ref{th-suff} and~\ref{th-unsuff}). We also provide descriptions of the $\mathbf f$-weighted $\alpha$-Riesz vector potentials of the solutions $\boldsymbol\lambda_{\mathbf A}^{\boldsymbol\xi}$, single out their characteristic properties, and analyze the supports of the $\lambda^i_{\mathbf A}$, $i\in I$ (see Theorems~\ref{desc-pot}, \ref{desc-sup} and~\ref{zone}). These results are illustrated in Examples~\ref{ex} and~\ref{ex2}. See also Section~\ref{sec:ext} for an extension of the theory to the case where $\xi^p\ne\infty$. The proofs of Theorems~\ref{th-suff}--\ref{zone} are given in Sections~\ref{sec-pr1} and~\ref{sec-pr2}; they are substantially based on Theorem~\ref{th-rel} which is a subject of the next section.

\section{Relations between minimum $\alpha$-Riesz and $\alpha$-Green energy problems}\label{deep}

Throughout this section, $\mathbf A$, $\mathbf a$, $\mathbf f$ and $\boldsymbol\xi$ are as indicated at the beginning of Section~\ref{sec-form}, except for (\ref{cggen}) which is temporarily not required. The aim of Theorem~\ref{th-rel} below is to establish a relationship between, on the one hand, the solvability (or the non-solvability) of Problem~\ref{pr2gen} for $\mathbb R^n$, $\kappa_\alpha$, $\mathbf A$, $\mathbf a$, $\mathbf f$, $\boldsymbol\xi$ and, on the other hand, that for $D$, $g=g^\alpha_D$, $\mathbf A^+$, $\mathbf a^+$, $\mathbf f^+$ and $\boldsymbol\xi^+$, where
\[\mathbf A^+:=(A_j)_{j\in I^+}, \ \mathbf a^+:=(a_j)_{j\in I^+}, \ \mathbf f^+:=(f_j|_D)_{j\in I^+}, \ \boldsymbol\xi^+:=(\xi^j)_{j\in I^+}.\]
(Note that $\mathbf A^+$ is a standard condenser in $X=D$ consisting of only positive plates.) Observe that since for every given $j\in I^+$ we have
\begin{equation}\label{inclusion}\mathfrak M^+(A_j;\mathbb R^n)\subset\mathfrak M^+(A_j;D),\end{equation}
the measure $\xi^j$ can certainly be thought of as an element of $\mathfrak C(A_j;D)$.

For any $\boldsymbol\mu=(\mu^i)_{i\in I}\in\mathfrak M^+(\mathbf A;\mathbb R^n)$ write $\boldsymbol\mu^+:=(\mu^j)_{j\in I^+}$; then $\boldsymbol\mu^+$ belongs to $\mathfrak M^+(\mathbf A^+;D)$ by (\ref{inclusion}). If moreover $\kappa_\alpha(\boldsymbol\mu,\boldsymbol\mu)<\infty$, then $\boldsymbol\mu^+$ belongs to $\mathcal E^+_\alpha(\mathbf A^+;\mathbb R^n)$, as well as to $\mathcal E_g^+(\mathbf A^+;D)$, the latter being clear from~(\ref{l3-1}).

\begin{theorem}\label{th-rel}Under the just mentioned assumptions on\/ $\mathbf A$, $\mathbf a$, $\mathbf f$ and\/ $\boldsymbol\xi$,
\begin{equation}\label{equality}G^{\boldsymbol\xi}_{\alpha,\mathbf f}(\mathbf A,\mathbf a;\mathbb R^n)=G^{\boldsymbol\xi^+}_{g,\mathbf f^+}(\mathbf A^+,\mathbf a^+;D).\end{equation}
If moreover these\/ {\rm(}equal\/{\rm)} extremal values are finite, then\/ $\mathfrak S^{\boldsymbol{\xi}}_{\alpha,\mathbf f}(\mathbf A,\mathbf a;\mathbb R^n)$ is nonempty if and only if so is\/ $\mathfrak S^{\boldsymbol{\xi^+}}_{g,\mathbf f^+}(\mathbf A^+,\mathbf a^+;D)$, and in the affirmative case
the following two assertions are equivalent for any\/ $\boldsymbol\lambda_{\mathbf A}=(\lambda_{\mathbf A}^i)_{i\in I}\in\mathfrak M^+(\mathbf A;\mathbb R^n)${\rm:}
\begin{itemize}
\item[{\rm(i)}] $\boldsymbol\lambda_{\mathbf A}\in\mathfrak S^{\boldsymbol{\xi}}_{\alpha,\mathbf f}(\mathbf A,\mathbf a;\mathbb R^n)$.
\item[{\rm(ii)}] $\boldsymbol\lambda_{\mathbf A}^+=(\lambda_{\mathbf A}^j)_{j\in I^+}\in\mathfrak S^{\boldsymbol{\xi^+}}_{g,\mathbf f^+}(\mathbf A^+,\mathbf a^+;D)$ and, in addition,
\begin{equation}\label{reprrr}
\lambda_{\mathbf A}^p=\Bigl(\sum_{j\in I^+}\,\lambda_{\mathbf A}^j\Bigr)'.
\end{equation}
\end{itemize}
\end{theorem}

\begin{proof} We begin by establishing the inequality
\begin{equation}\label{in1}G^{\boldsymbol\xi^+}_{g,\mathbf f^+}(\mathbf A^+,\mathbf a^+;D)\geqslant G^{\boldsymbol\xi}_{\alpha,\mathbf f}(\mathbf A,\mathbf a;\mathbb R^n).\end{equation}
Assuming $G^{\boldsymbol\xi^+}_{g,\mathbf f^+}(\mathbf A^+,\mathbf a^+;D)<\infty$, choose $\boldsymbol\mu=(\mu^j)_{j\in I^+}\in\mathcal E^{\boldsymbol\xi^+}_{g,\mathbf f^+}(\mathbf A^+,\mathbf a^+;D)$. Then, according to (\ref{enpos}) and (\ref{wen'}) with $X=D$ and $\kappa=g$,
\[G_{g,\mathbf f^+}(\boldsymbol\mu)=g(\boldsymbol\mu,\boldsymbol\mu)+2\langle\mathbf f^+,\boldsymbol\mu\rangle=
\|R_{\mathbf A^+}\boldsymbol\mu\|^2_g+2\langle\mathbf f^+,\boldsymbol\mu\rangle.\]
Being bounded, each of the $\mu^j$, $j\in I^+$, is extendible (see Section~\ref{sec:RG}). Furthermore, the extension in question has finite $\alpha$-Riesz energy, for so does the extension of the constraint $\xi^j$ by (\ref{constr1}). Applying (\ref{eq1-2-hen}) to $R_{\mathbf A^+}\boldsymbol\mu\in\mathcal E^+_\alpha(A^+;\mathbb R^n)$ in place of $\mu$, we thus get
\[G_{g,\mathbf f^+}(\boldsymbol\mu)=\|R_{\mathbf A^+}\boldsymbol\mu-(R_{\mathbf A^+}\boldsymbol\mu)'\|^2_\alpha+2\langle\mathbf f^+,\boldsymbol\mu\rangle.\]
Since $A_p$ $({}=D^c)$ is not $\alpha$-thin at infinity, we conclude from (\ref{t-mass'}) and (\ref{ap}) that $\bigl(R_{\mathbf A^+}\boldsymbol\mu\bigr)'\in\mathcal E_\alpha^+(A_p,a_p;\mathbb R^n)$, and therefore $\tilde{\boldsymbol\mu}=(\tilde{\mu}^i)_{i\in I}\in\mathcal E^{\boldsymbol\xi}_\alpha(\mathbf A,\mathbf a;\mathbb R^n)$ where
\begin{equation}\label{111}\tilde{\boldsymbol\mu}^+=\boldsymbol\mu\text{ \ and \ }\tilde{\mu}^p=\bigl(R_{\mathbf A^+}\boldsymbol\mu\bigr)'=\Bigl(\sum_{j\in I^+}\,\mu^j\Bigr)'.\end{equation}
Here we have used the (permanent) assumption that $\xi^p=\infty$. Furthermore,
\begin{equation*}\label{2}\langle\mathbf f,\tilde{\boldsymbol\mu}\rangle=\langle\mathbf f^+,\boldsymbol\mu\rangle<\infty,\end{equation*}
the equality being valid because $f^p=0$ n.e.\ on $A_p$ (see Section~\ref{sec-def}), hence $\tilde{\mu}^p$-a.e.\ by \cite[Lemma~2.3.1]{F1}, and also because $D^c$ is $\mu^j$-negligible for every $j\in I^+$. Thus
$\tilde{\boldsymbol\mu}\in\mathcal E^{\boldsymbol\xi}_{\alpha,\mathbf f}(\mathbf A,\mathbf a;\mathbb R^n)$ and $G_{\alpha,\mathbf f}(\tilde{\boldsymbol\mu})=G_{g,\mathbf f^+}(\boldsymbol\mu)$, where the latter relation follows from the preceding three displays. This yields
\[G_{g,\mathbf f^+}(\boldsymbol\mu)=G_{\alpha,\mathbf f}(\tilde{\boldsymbol\mu})\geqslant G^{\boldsymbol\xi}_{\alpha,\mathbf f}(\mathbf A,\mathbf a;\mathbb R^n),\]
which in view of the arbitrary choice of $\boldsymbol\mu\in\mathcal E^{\boldsymbol\xi^+}_{g,\mathbf f^+}(\mathbf A^+,\mathbf a^+;D)$ es\-tab\-lishes~(\ref{in1}).

On the other hand, in view of (\ref{l3-1}) and (\ref{plus}) for any $\boldsymbol\nu\in\mathcal E^{\boldsymbol\xi}_{\alpha,\mathbf f}(\mathbf A,\mathbf a;\mathbb R^n)$ we have $\boldsymbol\nu^+\in\mathcal E^{\boldsymbol\xi^+}_{g,\mathbf f^+}(\mathbf A^+,\mathbf a^+;D)$.
Thus, by (\ref{enpos}), (\ref{proj}), (\ref{eq1-2-hen}) and (\ref{plus}),
\begin{align}\notag G_{\alpha,\mathbf f}(\boldsymbol\nu)&=\kappa_\alpha(\boldsymbol\nu,\boldsymbol\nu)+2\langle\mathbf f^+,\boldsymbol\nu^+\rangle=
\|R_{\mathbf A}\boldsymbol\nu^+-\nu^p\|^2_\alpha+2\langle\mathbf f^+,\boldsymbol\nu^+\rangle\\{}&\geqslant
\|R_{\mathbf A}\boldsymbol\nu^+-(R_{\mathbf A}\boldsymbol\nu^+)'\|^2_\alpha+2\langle\mathbf f^+,\boldsymbol\nu^+\rangle=
\|R_{\mathbf A}\boldsymbol\nu^+\|^2_g+2\langle\mathbf f^+,\boldsymbol\nu^+\rangle\notag\\
{}&=g(\boldsymbol\nu^+,\boldsymbol\nu^+)+2\langle\mathbf f^+,\boldsymbol\nu^+\rangle=G_{g,\mathbf f^+}(\boldsymbol\nu^+)\geqslant G^{\boldsymbol\xi^+}_{g,\mathbf f^+}(\mathbf A^+,\mathbf a^+;D).\label{chain}
\end{align}
Since $\boldsymbol\nu\in\mathcal E^{\boldsymbol\xi}_{\alpha,\mathbf f}(\mathbf A,\mathbf a;\mathbb R^n)$ has been chosen arbitrarily, this together with (\ref{in1}) proves (\ref{equality}).

Now suppose that there exists $\boldsymbol\mu=(\mu^j)_{j\in I^+}\in\mathfrak S^{\boldsymbol{\xi^+}}_{g,\mathbf f^+}(\mathbf A^+,\mathbf a^+;D)$. Define $\tilde{\boldsymbol\mu}=(\tilde{\mu}^i)_{i\in I}$ as in (\ref{111}). Then the same arguments as those applied in the first paragraph of this proof enable us to see that $\tilde{\boldsymbol\mu}\in\mathcal E^{\boldsymbol\xi}_{\alpha,\mathbf f}(\mathbf A,\mathbf a;\mathbb R^n)$ and also that $G_{\alpha,\mathbf f}(\tilde{\boldsymbol\mu})=G_{g,\mathbf f^+}(\boldsymbol\mu)$. The latter yields
\[G_{\alpha,\mathbf f}(\tilde{\boldsymbol\mu})=G^{\boldsymbol\xi^+}_{g,\mathbf f^+}(\mathbf A^+,\mathbf a^+;D).\]
Substituting (\ref{equality}) into the last display shows that, actually, $\tilde{\boldsymbol\mu}\in\mathfrak S^{\boldsymbol{\xi}}_{\alpha,\mathbf f}(\mathbf A,\mathbf a;\mathbb R^n)$, which in view of the latter relation in (\ref{111}) proves that, indeed, (ii) implies~(i).

To establish the converse implication, assume that there is $\boldsymbol\nu=(\nu^i)_{i\in I}\in\mathfrak S^{\boldsymbol{\xi}}_{\alpha,\mathbf f}(\mathbf A,\mathbf a;\mathbb R^n)$. Then $\boldsymbol\nu^+\in\mathcal E^{\boldsymbol\xi^+}_{g,\mathbf f^+}(\mathbf A^+,\mathbf a^+;D)$ (see the second paragraph of the proof) and, in addition, (\ref{chain}) holds. Since for this $\boldsymbol\nu$ the first term in (\ref{chain}) equals $G^{\boldsymbol\xi}_{\alpha,\mathbf f}(\mathbf A,\mathbf a;\mathbb R^n)$, we conclude from (\ref{equality}) that all the inequalities in (\ref{chain}) are in fact equalities. This implies that $\boldsymbol\nu^+\in\mathfrak S^{\boldsymbol{\xi^+}}_{g,\mathbf f^+}(\mathbf A^+,\mathbf a^+;D)$ and also that $\nu^p=(R\boldsymbol\nu^+)'$, the latter being clear from (\ref{proj}).\end{proof}

\section{Main results}\label{s:main}
Throughout Section~\ref{s:main} we keep all the assumptions on $\mathbf A$, $\mathbf a$, $\mathbf f$ and $\boldsymbol\xi$ imposed at the beginning of Section~\ref{sec-form}, except for (\ref{cggen}).\footnote{Under the hypotheses of any of Theorems~\ref{th-unsuff}--\ref{zone}, (\ref{cggen}) holds in consequence of Lemma~\ref{suff-fin}.\label{f:aut}}

\subsection{Formulations of the main results}\label{form:main}

\begin{theorem}\label{th-suff} Assume now that\/ {\rm(\ref{cggen})} is fulfilled and, moreover,
\begin{equation}\label{boundd}\xi^j(A_j)<\infty\text{ \ for all \ }j\in I^+.\end{equation}
Then the class\/ $\mathfrak S^{\boldsymbol{\xi}}_{\alpha,\mathbf f}(\mathbf A,\mathbf a;\mathbb R^n)$ of the solutions to Problem\/~{\rm\ref{pr2gen}} is nonempty, and for any one of its elements\/ $\boldsymbol\lambda_{\mathbf A}^{\boldsymbol\xi}=(\lambda_{\mathbf A}^i)_{i\in I}$ we have\/
$\lambda_{\mathbf A}^p=\bigl(\sum_{j\in I^+}\,\lambda_{\mathbf A}^j\bigr)'$.\end{theorem}

Theorem~\ref{th-suff} is sharp in the sense that it no longer holds if requirement (\ref{boundd}) is omitted from its hypotheses (see the following Theorem~\ref{th-unsuff}).

\begin{theorem}\label{th-unsuff}Condition\/ {\rm(\ref{boundd})} is in general also necessary for the solvability of Problem\/~{\rm\ref{pr2gen}}. More precisely, suppose that\/ $I^+=\{1\}$, $c_\alpha(A_1)=\infty$ and that Case\/~{\rm II} holds with\/ $\zeta\geqslant0$. Then there is a constraint\/ $\xi^1\in\mathfrak C(A_1;\mathbb R^n)\cap\mathcal E_\alpha^+(A_1;\mathbb R^n)$ with\/ $\xi^1(A_1)=\infty$ such that
\begin{equation}\label{Gzero}G_{\alpha,\mathbf f}^{\boldsymbol\xi}(\mathbf A,\mathbf a;\mathbb R^n)=G_{g,f_1|_D}^{\xi^1}(A_1,a_1;D)=0;\end{equation}
hence\/ $G_{\alpha,\mathbf f}^{\boldsymbol\xi}(\mathbf A,\mathbf a;\mathbb R^n)$ cannot be an actual minimum because\/ $\mathbf 0\notin\mathcal E_{\alpha,\mathbf f}^{\boldsymbol\xi}(\mathbf A,\mathbf a;\mathbb R^n)$.
\end{theorem}

The following three assertions provide descriptions of the $\mathbf f$-weighted $\alpha$-Riesz potentials $W^{\boldsymbol\lambda^{\boldsymbol\xi}_{\mathbf A}}_{\alpha,\mathbf f}$, cf.\ (\ref{wpot'}), of the solutions $\boldsymbol\lambda^{\boldsymbol\xi}_{\mathbf A}=(\lambda_{\mathbf A}^i)_{i\in I}\in\mathfrak S^{\boldsymbol{\xi}}_{\alpha,\mathbf f}(\mathbf A,\mathbf a;\mathbb R^n)$ (whenever these exist), single out their characteristic properties, and analyze the supports of the $\lambda_{\mathbf A}^i$, $i\in I$.

\begin{theorem}\label{desc-pot}Let\/ {\rm(\ref{acirc})} hold, and let each\/ $f_j$, $j\in I^+$, be lower bounded on\/ $A_j$. Fix\/ $\boldsymbol\lambda_{\mathbf A}\in\mathcal E^{\boldsymbol\xi}_{\alpha,\mathbf f}(\mathbf A,\mathbf a;\mathbb R^n)$ {\rm(}which exists according to footnote\/~{\rm\ref{f:aut}}{\rm)}. Then the following two assertions are equivalent:
\begin{itemize}
\item[{\rm(i)}] $\boldsymbol\lambda_{\mathbf A}\in\mathfrak S^{\boldsymbol{\xi}}_{\alpha,\mathbf f}(\mathbf A,\mathbf a;\mathbb R^n)$.
\item[{\rm(ii)}] There exists a vector\/ $(c_j)_{j\in I^+}\in\mathbb R^{p-1}$ such that for all\/ $j\in I^+$
\begin{align}\label{b1}W^{\boldsymbol\lambda_{\mathbf A},j}_{\alpha,\mathbf f}&\geqslant c_j\quad(\xi^j-\lambda_{\mathbf A}^j)\text{-a.e.},\\
\label{b2}W^{\boldsymbol\lambda_{\mathbf A},j}_{\alpha,\mathbf f}&\leqslant c_j\quad\lambda_{\mathbf A}^j\text{-a.e.},
\end{align}
and in addition we have
\begin{equation}\label{reprrr'}
W^{\boldsymbol\lambda_{\mathbf A},p}_{\alpha,\mathbf f}=0\text{ \ n.e. on \ }A_p.
\end{equation}
If moreover Case\/ {\rm II} holds, then relation\/ {\rm(\ref{reprrr'})} actually holds for every\/ $W^{\boldsymbol\lambda_{\mathbf A},i}_{\alpha,\mathbf f}$, $i\in I$, and it takes now the form
\begin{equation}\label{reprrr1'}
W^{\boldsymbol\lambda_{\mathbf A},i}_{\alpha,\mathbf f}=0\text{ \ on \ }A_p\setminus I_{\alpha,A_p}, \ i\in I,
\end{equation}
where\/ $I_{\alpha,A_p}$ denotes the set of all\/ $\alpha$-irregular\/ {\rm(}boundary\/{\rm)} points of\/ $A_p$.
\end{itemize}
\end{theorem}

\begin{remark}{\rm The lower boundedness of $f_j$, $j\in I^+$, assumed in Theorem~\ref{desc-pot}, holds automatically provided that Case~I takes place. Furthermore, in Case~I relation (\ref{b2}) is equivalent to the following apparently stronger assertion:}
\[W^{\boldsymbol\lambda_{\mathbf A},j}_{\alpha,\mathbf f}\leqslant c_j\text{\rm \ on \ }S_D^{\lambda_{\mathbf A}^j}.\]\end{remark}

Let $\breve{Q}$ denote the {\it $c_\alpha$-reduced kernel\/} of $Q\subset\mathbb R^n$ \cite[p.~164]{L}, which is the set of all $x\in Q$ such that for any $r>0$ we have $c_\alpha\bigl(B(x,r)\cap Q\bigr)>0$.

For the sake of simplicity of formulation, in the following Theorem~\ref{desc-sup} we assume that in the case $\alpha=2$ the domain $D$ is {\it simply\/} connected.

\begin{theorem}\label{desc-sup}If a solution\/ $\boldsymbol\lambda^{\boldsymbol\xi}_{\mathbf A}=(\lambda_{\mathbf A}^i)_{i\in I}\in\mathfrak S^{\boldsymbol{\xi}}_{\alpha,\mathbf f}(\mathbf A,\mathbf a;\mathbb R^n)$ exists, then
\begin{equation}\label{lemma-desc-riesz}
S^{\lambda^p_{\mathbf A}}_{\mathbb R^n}=\left\{
\begin{array}{lll} \breve{A}_p & \text{ \ if \ } & \alpha<2,\\ \partial D  & \text{ \ if \ } & \alpha=2.\\ \end{array} \right.
\end{equation}
\end{theorem}

Assume now that $I^+=\{1\}$, $\mathbf a=\mathbf 1$, $\mathbf f=\mathbf 0$ and that there is a solution $\boldsymbol\lambda^{\boldsymbol\xi}_{\mathbf A}=(\lambda_{\mathbf A}^1,\lambda_{\mathbf A}^2)\in\mathfrak S^{\boldsymbol{\xi}}_{\alpha,\mathbf f}(\mathbf A,\mathbf 1;\mathbb R^n)$. Then, equivalently, $\lambda:=R_{\mathbf A}\boldsymbol\lambda^{\boldsymbol\xi}_{\mathbf A}=\lambda_{\mathbf A}^1-\lambda_{\mathbf A}^2$ is a solution to the minimum $\alpha$-Riesz energy problem
\begin{equation}\label{sc-pr}\inf\,\kappa_\alpha(\mu,\mu),\end{equation}
where $\mu$ ranges over all ({\it signed scalar\/} Radon) measures with $\mu^+\in\mathcal E^{\xi^1}_\alpha(A_1,1;\mathbb R^n)$ and $\mu^-\in\mathcal E^+_\alpha(A_2,1;\mathbb R^n)$. Since $\mathbf f=\mathbf 0$, we also see from (\ref{repp}) and (\ref{wpot'}) that
\begin{equation}\label{potl}\kappa_\alpha(\cdot,\lambda)=s_i\kappa_\alpha^{\boldsymbol\lambda^{\boldsymbol\xi}_{\mathbf A},i}(\cdot)=s_iW_{\alpha,\mathbf f}^{\boldsymbol\lambda^{\boldsymbol\xi}_{\mathbf A},i}(\cdot)\text{ \ n.e.\ on \ }\mathbb R^n, \ i=1,2.\end{equation}

\begin{theorem}\label{zone}With these assumptions and notations, we have
\begin{equation}\label{th}\kappa_\alpha(\cdot,\lambda)=\left\{
\begin{array}{lll} g^\alpha_D(\cdot,\lambda^+) & \text{n.e.\ on} & D,\\
0 & \text{on} & D^c\setminus I_{\alpha,D^c}.\\ \end{array} \right.
\end{equation}
Furthermore, assertion\/ {\rm(ii)} of Theorem\/~{\rm\ref{desc-pot}} holds, and\/ {\rm(\ref{b1})} and\/ {\rm(\ref{b2})} now take the form
\begin{align}\label{b21}\kappa_\alpha(\cdot,\lambda)&=c_1\text{ \ $(\xi^1-\lambda^+)$-a.e.},\\
\label{b11}\kappa_\alpha(\cdot,\lambda)&\leqslant c_1\text{ \ on \ }\mathbb R^n,
\end{align}
respectively, where\/ $0<c_1<\infty$. In addition, {\rm(\ref{b21})} and\/ {\rm(\ref{b11})} together with\/ $\kappa_\alpha(\cdot,\lambda)=0$ n.e.\ on\/ $D^c$ {\rm(}cf.\ {\rm(\ref{th})}{\rm)} determine uniquely the solution to the problem\/~{\rm(\ref{sc-pr})} among the admissible  measures\/ $\mu$. If moreover\/ $\kappa_\alpha(\cdot,\xi^1)$ is\/ {\rm(}finitely\/{\rm)} continuous on\/ $D$, then also
\begin{equation}\label{b21'}\kappa_\alpha(\cdot,\lambda)=c_1\text{ \ on \ }S^{\xi^1-\lambda^+}_D,\end{equation}
\begin{equation}\label{cg}c_{g_D^\alpha}\bigl(S_D^{\xi^1-\lambda^+}\bigr)<\infty.\end{equation}
Omitting now the requirement of the continuity of $\kappa_\alpha(\cdot,\xi^1)$, assume further that\/ $\alpha<2$ and\/ $m_n(D^c)>0$ where\/ $m_n$ is the\/ $n$-dimen\-sional Lebesgue measure. Then
\begin{equation}\label{pg}S_D^{\lambda^+}=S_D^{\xi^1},\end{equation}
\begin{equation}\label{supp}\kappa_\alpha(\cdot,\lambda)<c_1\text{ \ on \ }D\setminus S_D^{\xi^1}\quad\bigl({}=D\setminus S_D^{\lambda^+}\bigr).\end{equation}
\end{theorem}

\subsection{An extension of the theory}\label{sec:ext} Parallel with Problem~\ref{pr2gen} for a constraint $\boldsymbol\xi$ given by (\ref{constr1}) and acting only on measures concentrated on the positive plates $A_j$, $j\in I^+$, of the generalized condenser $\mathbf A=(A_i)_{i\in I}$, consider also Problem~\ref{pr2gen} for $\boldsymbol\sigma=(\sigma^i)_{i\in I}\in\mathfrak M^+(\mathbf A;\mathbb R^n)$ (in place of $\boldsymbol\xi$) defined as follows:
\begin{equation}\label{sp}\sigma^j=\xi^j\text{ \ for all \ }j\in I^+,\qquad\sigma^p\geqslant\Bigl(\sum_{j\in I^+}\,\sigma^j\Bigr)'\quad\Bigl[{}=\Bigl(\sum_{j\in I^+}\,\xi^j\Bigr)'\Bigr].\end{equation}
Since in consequence of (\ref{t-mass'}), (\ref{ap}) and (\ref{sp}) we have $\sigma^p(A_p)>a_p$, the measure $\boldsymbol\sigma$ thus defined can be thought of as an element of $\mathfrak C(\mathbf A;\mathbb R^n)$. In contrast to $\boldsymbol\xi$, {\it the constraint\/ $\boldsymbol\sigma$ is acting on all the components of\/} $\boldsymbol\mu\in\mathcal E^+_\alpha(\mathbf A,\mathbf a;\mathbb R^n)$. Also note that $\sigma^p(A_p)$ and $\kappa_\alpha(\sigma^p,\sigma^p)$ may both be infinite.

\begin{theorem}\label{l:eq}The following identity holds:
\begin{equation}\label{eq:eq}G_{\alpha,\mathbf{f}}^{\boldsymbol\sigma}(\mathbf A,\mathbf a;\mathbb R^n)=G_{\alpha,\mathbf{f}}^{\boldsymbol\xi}(\mathbf A,\mathbf a;\mathbb R^n).
\end{equation}
If moreover these\/ {\rm(}equal\/{\rm)} extremal values are finite, then Problem\/~{\rm\ref{pr2gen}} for\/ $\mathbf A$, $\mathbf a$, $\mathbf f$ and\/ $\boldsymbol\xi$ is solvable if and only if so is that for\/ $\mathbf A$, $\mathbf a$, $\mathbf f$ and\/ $\boldsymbol\sigma$, and in the affirmative case
\begin{equation}\label{l:id}\mathfrak S^{\boldsymbol{\sigma}}_{\alpha,\mathbf f}(\mathbf A,\mathbf a;\mathbb R^n)=\mathfrak S^{\boldsymbol{\xi}}_{\alpha,\mathbf f}(\mathbf A,\mathbf a;\mathbb R^n).\end{equation}
\end{theorem}

\begin{proof} Indeed, $G_{\alpha,\mathbf f}^{\boldsymbol\sigma}(\mathbf A,\mathbf a;\mathbb R^n)\geqslant G_{\alpha,\mathbf f}^{\boldsymbol\xi}(\mathbf A,\mathbf a;\mathbb R^n)$ follows from the relation
\begin{equation}\label{pr:inclus}\mathcal E_{\alpha,\mathbf f}^{\boldsymbol\sigma}(\mathbf A,\mathbf a;\mathbb R^n)\subset\mathcal E_{\alpha,\mathbf{f}}^{\boldsymbol\xi}(\mathbf A,\mathbf a;\mathbb R^n).\end{equation}
To establish the converse inequality, assume that $G_{\alpha,\mathbf f}^{\boldsymbol\xi}(\mathbf A,\mathbf a;\mathbb R^n)<\infty$ and fix a minimizer
$\boldsymbol\nu\in\mathcal E_{\alpha,\mathbf f}^{\boldsymbol\xi}(\mathbf A,\mathbf a;\mathbb R^n)$. Define $\tilde{\boldsymbol\nu}=(\tilde{\nu}^i)_{i\in I}\in\mathfrak M^+(\mathbf A;\mathbb R^n)$ by the equalities
\begin{equation}\label{pr:l:1}\tilde{\boldsymbol\nu}^+=\boldsymbol\nu^+\text{ \  and \ }\tilde{\nu}^p=(R_{\mathbf A}\boldsymbol{\nu}^+)'.\end{equation}
Clearly $\tilde{\boldsymbol\nu}\in\mathcal E^+_\alpha(\mathbf A;\mathbb R^n)$, and moreover $\tilde{\boldsymbol\nu}\in\mathcal E^+_\alpha(\mathbf A,\mathbf a;\mathbb R^n)$ which follows from (\ref{t-mass'}), (\ref{ap}) and (\ref{pr:l:1}) since $A_p$ is not $\alpha$-thin at infinity. By the linearity of balayage and (\ref{sp}) we actually have $\tilde{\boldsymbol\nu}\in\mathcal E_\alpha^{\boldsymbol\sigma}(\mathbf A,\mathbf a;\mathbb R^n)$, and finally $\tilde{\boldsymbol\nu}\in\mathcal E_{\alpha,\mathbf f}^{\boldsymbol\sigma}(\mathbf A,\mathbf a;\mathbb R^n)$ by (\ref{plus}). In consequence of (\ref{proj}), (\ref{plus}) and (\ref{pr:l:1}) we therefore obtain
\begin{align*}G_{\alpha,\mathbf f}(\boldsymbol\nu)&=\kappa_\alpha(\boldsymbol\nu,\boldsymbol\nu)+2\langle\mathbf{f}^+,\boldsymbol{\nu}^+\rangle=
\|R_{\mathbf A}\boldsymbol{\nu}^+-\nu^p\|^2_\alpha+2\langle\mathbf{f}^+,\boldsymbol{\nu}^+\rangle\\
\notag{}&\geqslant\|R_{\mathbf A}\boldsymbol{\nu}^+-(R_{\mathbf A}\boldsymbol{\nu}^+)'\|^2_\alpha
+2\langle\mathbf{f}^+,\boldsymbol{\nu}^+\rangle=G_{\alpha,\mathbf f}(\tilde{\boldsymbol\nu})\geqslant G_{\alpha,\mathbf f}^{\boldsymbol\sigma}(\mathbf A,\mathbf a;\mathbb R^n),\end{align*}
which establishes (\ref{eq:eq}) in view of the arbitrary choice of $\boldsymbol\nu\in\mathcal E_{\alpha,\mathbf f}^{\boldsymbol\xi}(\mathbf A,\mathbf a;\mathbb R^n)$.

Assuming now that (\ref{cggen}) holds, we proceed to prove (\ref{l:id}). The inclusion $\mathfrak S^{\boldsymbol{\sigma}}_{\alpha,\mathbf f}(\mathbf A,\mathbf a;\mathbb R^n)\subset\mathfrak S^{\boldsymbol{\xi}}_{\alpha,\mathbf f}(\mathbf A,\mathbf a;\mathbb R^n)$ is obvious because of (\ref{eq:eq}) and (\ref{pr:inclus}). To establish the converse inclusion, fix $\boldsymbol\lambda=(\lambda^i)_{i\in I}\in\mathfrak S^{\boldsymbol{\xi}}_{\alpha,\mathbf f}(\mathbf A,\mathbf a;\mathbb R^n)$. Then by (\ref{reprrr}) we have $\lambda^p=(R_{\mathbf A}\boldsymbol{\lambda}^+)'$, and in the same manner as in the preceding paragraph we obtain $\boldsymbol\lambda\in\mathcal E_{\alpha,\mathbf f}^{\boldsymbol\sigma}(\mathbf A,\mathbf a;\mathbb R^n)$. Hence $\boldsymbol\lambda$ also belongs to $\mathfrak S^{\boldsymbol{\sigma}}_{\alpha,\mathbf f}(\mathbf A,\mathbf a;\mathbb R^n)$, for $G_{\alpha,\mathbf f}(\boldsymbol\lambda)=G_{\alpha,\mathbf f}^{\boldsymbol\xi}(\mathbf A,\mathbf a;\mathbb R^n)=G_{\alpha,\mathbf f}^{\boldsymbol\sigma}(\mathbf A,\mathbf a;\mathbb R^n)$ by~(\ref{eq:eq}).
\end{proof}

Thus, {\it the theory of minimum\/ $\alpha$-Riesz energy problems with a constraint\/ $\boldsymbol\xi$ given by\/ {\rm(\ref{constr1})} and acting only on measures concentrated on the\/ $A_j$, $j\in I^+$, developed in Section\/~{\rm\ref{form:main}}, remains valid in its full generality for the constraint\/ $\boldsymbol\sigma$, defined by\/ {\rm(\ref{sp})} and acting on all the components of\/} $\boldsymbol\mu\in\mathcal E^+_\alpha(\mathbf A,\mathbf a;\mathbb R^n)$.

\section{Proofs of Theorems~\ref{th-suff} and~\ref{th-unsuff}}\label{sec-pr1}

Observe that, if Case II takes place, then
\begin{equation}\label{IIgg}\zeta\in\mathcal E_g(D),\end{equation}
\begin{equation}\label{IIg}f_j=\kappa_\alpha(\cdot,\zeta-\zeta')=g(\cdot,\zeta)\text{ \ $c_g$-n.e.\ on \ }D\text{ \ for all \ }j\in I^+.\end{equation}
Indeed, (\ref{IIgg}) is obvious by (\ref{l3-1}), and (\ref{IIg}) holds by Lemma~\ref{l-hatg} and footnote~\ref{RG}. By (\ref{IIgg}) and (\ref{IIg}), in Case~II for every $\boldsymbol{\nu}\in\mathcal E^+_g(\mathbf A^+;D)$ we get
\begin{align}\notag G_{g,\mathbf f^+}(\boldsymbol{\nu})&=\|R_{\mathbf A^+}\boldsymbol{\nu}\|^2_g+2\sum_{j\in I^+}\,g(\zeta,\nu^j)\\
{}&=\|R_{\mathbf A^+}\boldsymbol{\nu}\|^2_g+2g(\zeta,R_{\mathbf A^+}\boldsymbol{\nu})=
\|R_{\mathbf A^+}\boldsymbol{\nu}+\zeta\|^2_g-\|\zeta\|_g^2.\label{est:bel}\end{align}

\subsection{Proof of Theorem~\ref{th-suff}} By Theorem~\ref{th-rel}, Theorem~\ref{th-suff} will be proved once we have established the following assertion.

\begin{theorem}\label{exists}Under the assumptions of Theorem\/~{\rm\ref{th-suff}}, Problem\/~{\rm\ref{pr2gen}} for\/ $D$, $g$, $\mathbf A^+$, $\mathbf a^+$, $\mathbf f^+$ and\/ $\boldsymbol\xi^+$ is solvable, i.e.\ there is\/ $\boldsymbol\mu\in\mathcal E_{g,\mathbf{f}^+}^{\boldsymbol\xi^+}(\mathbf A^+,\mathbf a^+;D)$ with
\[G_{g,\mathbf f^+}(\boldsymbol\mu)=G_{g,\mathbf{f}^+}^{\boldsymbol\xi^+}(\mathbf A^+,\mathbf a^+;D).\]
\end{theorem}

\begin{proof} Note that Problem~\ref{pr2gen} for $D$, $g$, $\mathbf A^+$, $\mathbf a^+$, $\mathbf f^+$ and $\boldsymbol\xi^+$ makes sense since by assumption (\ref{cggen}) and identity (\ref{equality}) we have
\begin{equation}\label{g-c-plusinf}G_{g,\mathbf{f}^+}^{\boldsymbol\xi^+}(\mathbf A^+,\mathbf a^+;D)<\infty.\end{equation}
Actually, $G_{g,\mathbf{f}^+}^{\boldsymbol\xi^+}(\mathbf A^+,\mathbf a^+;D)$ is finite, which is clear from (\ref{equality}) and footnote~\ref{foot-G-finite}.
In view of (\ref{g-c-plusinf}), there is a sequence $\{\boldsymbol\mu_k\}_{k\in\mathbb N}\subset\mathcal E_{g,\mathbf{f}^+}^{\boldsymbol\xi^+}(\mathbf A^+,\mathbf a^+;D)$ such that
\begin{equation}\label{min-seq}\lim_{k\to\infty}\,G_{g,\mathbf{f}^+}(\boldsymbol\mu_k)=G_{g,\mathbf{f}^+}^{\boldsymbol\xi^+}(\mathbf A^+,\mathbf a^+;D).\end{equation}
Since the $\alpha$-Green kernel $g$ satisfies the energy principle \cite[Theorem~4.9]{FZ}, $\mathcal E_g(D)$ forms a pre-Hilbert space with the inner product $g(\nu,\nu_1)$ and the energy norm $\|\nu\|_g=\sqrt{g(\nu,\nu)}$. Furthermore, $\mathcal E_{g,\mathbf{f}^+}^{\boldsymbol\xi^+}(\mathbf A^+,\mathbf a^+;D)$ is a convex cone and $R_{\mathbf A^+}$ is an isometric mapping between the semimetric space $\mathcal E_g^+(\mathbf A^+;D)$ and its $R_{\mathbf A^+}$-im\-age into $\mathcal E_g(D)$ (see Theorem~\ref{lemma:semimetric}). We are therefore able to apply to the set $\{\boldsymbol\mu_k:\ k\in\mathbb N\}$ arguments similar to those in the proof of Lemma~\ref{lemma:unique:}, and we get
\begin{equation*}0\leqslant\|R_{\mathbf A^+}\boldsymbol{\mu}_k-R_{\mathbf A^+}\boldsymbol{\mu}_\ell\|^2_g\leqslant-
4G^{\boldsymbol{\xi}^+}_{g,\mathbf f^+}(\mathbf A^+,\mathbf a^+;D)+2G_{g,\mathbf f^+}(\boldsymbol{\mu}_k)+2G_{g,\mathbf f^+}(\boldsymbol{\mu}_\ell).\end{equation*}
Letting here $k,\ell\to\infty$ and combining the relation thus obtained with (\ref{min-seq}), we see in view of the finiteness of $G_{g,\mathbf{f}^+}^{\boldsymbol\xi^+}(\mathbf A^+,\mathbf a^+;D)$ that $\{R_{\mathbf A^+}\boldsymbol\mu_k\}_{k\in\mathbb N}$ forms a strong Cauchy sequence in the metric space $\mathcal E^+_g(D)$. In particular, this implies
\begin{equation}\label{M}\sup_{k\in\mathbb N}\,\|R_{\mathbf A^+}\boldsymbol\mu_k\|_g<\infty.\end{equation}

Since the sets $A_j$, $j\in I^+$, are (relatively) closed in $D$, the cones $\mathfrak M^{\xi^j}(A_j;D)$, $j\in I^+$, are vaguely closed in $\mathfrak M(D)$, and therefore $\mathfrak M^{\boldsymbol\xi^+}(\mathbf A^+;D)$ is vaguely closed in $\mathfrak M(D)^{p-1}$ (cf.\ Definition~\ref{def-vague}).
Furthermore, $\mathfrak M^{\boldsymbol\xi^+}(\mathbf A^+,\mathbf a^+;D)$ is vaguely bounded, hence vaguely relatively compact by Lemma~\ref{l:vague:c}. Thus, there is a vague cluster point $\boldsymbol{\mu}$ of the sequence $\{\boldsymbol\mu_k\}_{k\in\mathbb N}$ chosen above, which belongs to $\mathfrak M^{\boldsymbol\xi^+}(\mathbf A^+;D)$. Passing to a subsequence and changing notations, we assume that
\begin{equation}\label{vag-conv}\boldsymbol\mu_k\to\boldsymbol{\mu}\text{ \ vaguely as \ }k\to\infty.\end{equation}
We assert that the $\boldsymbol{\mu}$ is a solution to Problem~\ref{pr2gen} for $D$, $g$, $\mathbf A^+$, $\mathbf a^+$, $\mathbf f^+$ and $\boldsymbol\xi^+$.

Fix $j\in I^+$. Applying Lemma~\ref{lemma-semi} to $1_D\in\Psi(D)$, we obtain from (\ref{vag-conv})
\begin{equation*}\label{eq}\mu^j(D)\leqslant\lim_{k\to\infty}\,\mu_k^j(D)=a_j.\end{equation*}
We proceed by showing that the inequality here is in fact an equality, and hence
\begin{equation}\label{solution1}\boldsymbol{\mu}\in\mathfrak M^{\boldsymbol\xi^+}(\mathbf A^+,\mathbf a^+;D).\end{equation}
Consider an exhaustion of $A_j$ by an increasing sequence of compact sets $K_\ell\subset A_j$, $\ell\to\infty$. Since each $-1_{K_\ell}\in\Psi(D)$, we get from Lemma~\ref{lemma-semi} \begin{align*}a_j&\geqslant\langle 1_D,\mu^j\rangle=\lim_{\ell\to\infty}\,\langle 1_{K_\ell},\mu^j\rangle\geqslant\lim_{\ell\to\infty}\,\limsup_{k\to\infty}\,\langle 1_{K_\ell},\mu_{k}^j\rangle\\&{}
=a_j-\lim_{\ell\to\infty}\,\liminf_{k\to\infty}\,\langle 1_{A_j\setminus K_\ell},\mu_k^j\rangle.\end{align*}
Thus (\ref{solution1}) will follow once we show that
\begin{equation}\label{g0}\lim_{\ell\to\infty}\,\liminf_{k\to\infty}\,\langle 1_{A_j\setminus K_\ell},\mu_k^j\rangle=0.\end{equation}
By (\ref{boundd}),
\[\infty>\xi^j(D)=\lim_{\ell\to\infty}\,\langle 1_{K_\ell},\xi^j\rangle\]
and therefore
\[\lim_{\ell\to\infty}\,\langle 1_{A_j\setminus K_\ell},\xi^j\rangle=0.\]
Combined with
\[\langle 1_{A_j\setminus K_\ell},\mu_k^j\rangle\leqslant\langle 1_{A_j\setminus K_\ell},\xi^j\rangle\text{ \ for all \ }k\in\mathbb N,\]
this implies (\ref{g0}), and hence (\ref{solution1}).

Furthermore, as $R_{\mathbf A^+}\boldsymbol\mu_k\to R_{\mathbf A^+}\boldsymbol\mu$ vaguely in $\mathfrak M^+(D)$,  \cite[Chapitre~III, Section~5, Exercice~5]{B2} implies that $R_{\mathbf A^+}\boldsymbol\mu_k\otimes R_{\mathbf A^+}\boldsymbol\mu_k\to R_{\mathbf A^+}\boldsymbol\mu\otimes R_{\mathbf A^+}\boldsymbol\mu$ vaguely in $\mathfrak M^+(D\times D)$. Lemma~\ref{lemma-semi} with $X=D\times D$ and $\psi=g$ therefore yields
\[g(R_{\mathbf A^+}\boldsymbol\mu,R_{\mathbf A^+}\boldsymbol\mu)\leqslant\liminf_{k\to\infty}\,\|R_{\mathbf A^+}\boldsymbol\mu_k\|^2_g<\infty,\]
the latter holds by (\ref{M}). Together with (\ref{enpos}) and (\ref{solution1}) this gives
$\boldsymbol\mu\in\mathcal E_g^{\boldsymbol\xi^+}(\mathbf A^+,\mathbf a^+;D)$. Since $G_{g,\mathbf f^+}(\boldsymbol\mu)>-\infty$, $\boldsymbol{\mu}\in\mathfrak S^{\boldsymbol\xi^+}_{g,\mathbf f^+}(\mathbf A^+,\mathbf a^+;D)$ will be established once we have shown that
\begin{equation}\label{incl3}G_{g,\mathbf f^+}(\boldsymbol\mu)\leqslant\lim_{k\to\infty}\,G_{g,\mathbf f^+}(\boldsymbol\mu_k).\end{equation}

Since the kernel $g$ is perfect \cite[Theorem~4.11]{FZ}, the sequence $\{R_{\mathbf A^+}\boldsymbol\mu_k\}_{k\in\mathbb N}$, being strong Cauchy in $\mathcal E^+_g(D)$ and vaguely convergent to $R_{\mathbf A^+}\boldsymbol\mu$, converges to the same limit strongly in $\mathcal E_g^+(D)$, i.e.
\[\lim_{k\to\infty}\,\|R_{\mathbf A^+}\boldsymbol\mu_k-R_{\mathbf A^+}\boldsymbol\mu\|_g=0,\]
which in view of (\ref{isom}) is equivalent to the relation
\begin{equation}\label{conv-str}\lim_{k\to\infty}\,\|\boldsymbol\mu_k-\boldsymbol\mu\|_{\mathcal E^+_g(\mathbf A^+;D)}=0.\end{equation}

Also note that the mapping $\boldsymbol\nu\mapsto G_{g,\mathbf f^+}(\boldsymbol\nu)$ is vaguely l.s.c., resp.\ strongly continuous, on $\mathcal E^+_{g,\mathbf f^+}(\mathbf A^+;D)$ if Case~I, resp.\ Case~II, takes place. In fact, since $g(\boldsymbol\nu,\boldsymbol\nu)$ is vaguely l.s.c.\ on $\mathcal E^+_g(\mathbf A^+;D)$, the former assertion follows from Lemma~\ref{lemma-semi}. As for the latter assertion, it is obvious by (\ref{est:bel}). In view of this observation, (\ref{vag-conv}) and (\ref{conv-str}) result in~(\ref{incl3}).\end{proof}

\begin{corollary}Suppose that the assumptions of Theorem\/~{\rm\ref{th-suff}} are fulfilled. Then the\/ {\rm(}nonempty\/{\rm)} class\/ $\mathfrak S^{\boldsymbol\xi^+}_{g,\mathbf f^+}(\mathbf A^+,\mathbf a^+;D)$ of all solutions to Problem\/~{\rm\ref{pr2gen}} for\/ $D$, $g$, $\mathbf A^+$, $\mathbf a^+$, $\mathbf f^+$ and\/ $\boldsymbol\xi^+$ is vaguely compact in\/ $\mathfrak M(D)^{p-1}$.\end{corollary}

\begin{proof}According to Lemma~\ref{lemma:unique:}, any solutions $\boldsymbol\mu_k\in\mathfrak S^{\boldsymbol\xi^+}_{g,\mathbf f^+}(\mathbf A^+,\mathbf a^+;D)$, $k\in\mathbb N$, form a strong Cauchy sequence in $\mathcal E^+_g(\mathbf A^+;D)$. Furthermore, the set $\{\boldsymbol\mu_k:\ k\in\mathbb N\}$ is vaguely closed and vaguely relatively compact in $\mathfrak M(D)^{p-1}$ (see Section~\ref{sec:vague} with $X=D$). Therefore in the same manner as in the proof of Theorem~\ref{exists} one can see that any vague cluster point of $\{\boldsymbol\mu_k\}_{k\in\mathbb N}$ belongs to $\mathfrak S^{\boldsymbol\xi^+}_{g,\mathbf f^+}(\mathbf A^+,\mathbf a^+;D)$.\end{proof}

\subsection{Proof of Theorem~\ref{th-unsuff}} Assume that the requirements of the latter part of the theorem are fulfilled. By Theorem~\ref{th-rel}, the former equality in (\ref{Gzero}) holds. Furthermore, since Case II with $\zeta\geqslant0$ takes place, we get from (\ref{IIgg}) and (\ref{IIg})
\begin{equation}\label{caseII}G_{g,f_1|_D}(\nu)=\|\nu\|_g^2+2g(\zeta,\nu)\in[0,\infty)\text{ \ for all \ }\nu\in\mathcal E^+_g(A_1;D).\end{equation}

Consider numbers $r_\ell>0$, $\ell\in\mathbb N$, such that $r_\ell\uparrow\infty$ as $\ell\to\infty$, and write $B_{r_\ell}:=B(0,r_\ell)$, $A_{1,r_\ell}:=A_1\cap B_{r_\ell}$.
Since $c_\alpha(A_1)=\infty$ by assumption and since $c_\alpha(B_{r_\ell})<\infty$ for every $\ell\in\mathbb N$, we infer from the subadditivity of $c_\alpha(\cdot)$ on universally measurable sets \cite[Lemma~2.3.5]{F1} that $c_\alpha(A_1\setminus B_{r_\ell})=\infty$. Hence for every $\ell\in\mathbb N$ there is $\xi_\ell\in\mathcal E_\alpha^+(A_1\setminus B_{r_\ell},a_1;\mathbb R^n)$ of compact support $S_D^{\xi_\ell}$ such that
\begin{equation}\label{to0}\|\xi_\ell\|_\alpha\leqslant\ell^{-2}.\end{equation}
Clearly, the $r_\ell$ can be chosen successively so that $A_{1,r_\ell}\cup S^{\xi_\ell}_D\subset A_{1,r_{\ell+1}}$. Any compact set $K\subset\mathbb R^n$ is contained in a ball $B_{r_{\ell_0}}$ with $\ell_0$ large enough, and hence $K$ has points in common with only finitely many $S^{\xi_\ell}_D$. Therefore, $\xi^1$ defined by
\[\xi^1(\varphi):=\sum_{\ell\in\mathbb N}\,\xi_\ell(\varphi)\text{ \ for any \ }\varphi\in C_0(\mathbb R^n)\]
is a positive Radon measure on $\mathbb R^n$ carried by $A_1$. Furthermore, $\xi^1(A_1)=\infty$ and $\xi^1\in\mathcal E^+_\alpha(\mathbb R^n)$. To prove the latter, note that  $\eta_k:=\xi_1+\dots+\xi_k\in\mathcal E^+_\alpha(\mathbb R^n)$ in view of (\ref{to0}) and the triangle inequality in $\mathcal E_\alpha(\mathbb R^n)$. Also observe that $\eta_k\to\xi^1$ vaguely because for any $\varphi\in C_0(\mathbb R^n)$ there is $k_0$ such that
$\xi^1(\varphi)=\eta_k(\varphi)$ for all $k\geqslant k_0$. As $\|\eta_k\|_\alpha\leqslant L:=\sum_{\ell\in\mathbb N}\,\ell^{-2}<\infty$ for all $k\in\mathbb N$, Lemma~\ref{lemma-semi} with $X=A_1\times A_1$ and $\psi=\kappa_\alpha|_{A_1\times A_1}$ yields $\|\xi^1\|_\alpha\leqslant L$.

Each $\xi_\ell$ belongs to $\mathcal E^+_g(A_1,a_1;D)$ and moreover, by (\ref{g-ineq}) and~(\ref{to0}),
\begin{equation}\label{estell}\|\xi_\ell\|_g\leqslant\|\xi_\ell\|_\alpha\leqslant\ell^{-2}.\end{equation}
As Case II with $\zeta\geqslant0$ takes place, $\xi_\ell\in\mathcal E_{g,f_1|_D}^{\xi^1}(A_1,a_1;D)$ for all $\ell\in\mathbb N$ by (\ref{caseII}). Therefore, by the Cauchy--Schwarz (Bunyakovski) inequality in $\mathcal E_{g}(D)$,
\[0\leqslant G_{g,f_1|_D}^{\xi^1}(A_1,a_1;D)\leqslant\lim_{\ell\to\infty}\,\bigl[\|\xi_\ell\|_g^2+2g(\zeta,\xi_\ell)\bigr]\leqslant2\|\zeta\|_g\lim_{\ell\to\infty}\,\|\xi_\ell\|_g=0,\]
where the first and the second inequalities hold by (\ref{caseII}), and the third inequality and the equality are valid by (\ref{estell}). Hence $G_{g,f_1|_D}^{\xi^1}(A_1,a_1;D)=0$, and the theorem follows.

\section{Proofs of Theorems~\ref{desc-pot}, \ref{desc-sup} and  \ref{zone}}\label{sec-pr2}

Throughout this section we maintain all the requirements on $\mathbf A$, $\mathbf a$, $\mathbf f$, and $\boldsymbol\xi$ imposed at the beginning of Section~\ref{sec-form}, except for (\ref{cggen}) which follows automatically from the hypotheses of the assertions under proving in view of Lemma~\ref{suff-fin}.

\subsection{Proof of Theorem~\ref{desc-pot}} Fix $\boldsymbol\lambda_{\mathbf A}\in\mathcal E^+_{\alpha,\mathbf f}(\mathbf A,\mathbf a;\mathbb R^n)$. Then each $\lambda_{\mathbf A}^i$, $i\in I$, has finite $\alpha$-Riesz energy, and hence it is $c_\alpha$-absolutely continuous. Note that, since $f_p=0$ n.e.\ on $A_p$, (\ref{reprrr'}) can alternatively be written as $\kappa_\alpha^{\boldsymbol\lambda_{\mathbf A},p}=0$ n.e.\ on $A_p$, which by (\ref{repp}) (with $X=\mathbb R^n$ and $\kappa=\kappa_\alpha$) is equivalent to the relation
\[\kappa_\alpha(\cdot,R_{\mathbf A}\boldsymbol\lambda_{\mathbf A}^+-\lambda^p)=0\text{ \ n.e.~on \ }A_p.\]
In view of the characteristic property (\ref{bal-eq}) of the swept measures, this shows that for the given $\boldsymbol\lambda_{\mathbf A}$, (\ref{reprrr'}) and (\ref{reprrr}) are equivalent. On account of Theorem~\ref{th-rel}, we thus see that {\it when proving the equivalence of assertions\/ {\rm(i)} and\/ {\rm(ii)} of Theorem\/~{\rm\ref{desc-pot}}, there is no loss of generality in assuming\/ $\boldsymbol\lambda_{\mathbf A}$ to satisfy\/~{\rm(\ref{reprrr})}}.

Substituting (\ref{reprrr}) into (\ref{repp}), we therefore get for every $i\in I$
\begin{equation}\label{Wi1}\kappa_\alpha^{\boldsymbol\lambda_{\mathbf A},i}(\cdot)=s_i\kappa_\alpha\bigl(\cdot,R_{\mathbf A}\boldsymbol\lambda_{\mathbf A}^+-(R_{\mathbf A}\boldsymbol\lambda_{\mathbf A}^+)'\bigr)\text{ \ n.e.\ on \ }\mathbb R^n.\end{equation}
In particular, for every $j\in I^+$ we have
\begin{equation}\label{Wi}\kappa_\alpha^{\boldsymbol\lambda_{\mathbf A},j}(\cdot)=g(\cdot,R_{\mathbf A^+}\boldsymbol\lambda_{\mathbf A}^+)=g^{\boldsymbol\lambda_{\mathbf A}^+,j}(\cdot)\text{ \ n.e.\ on \ }D\end{equation}
and hence, by (\ref{wpot'}),
\[W_{\alpha,\mathbf{f}}^{\boldsymbol\lambda_{\mathbf A},j}=W_{g,\mathbf{f}^+}^{\boldsymbol\lambda_{\mathbf A}^+,j}\text{ \ n.e.\ on \ }D.\]
(Note that (\ref{Wi}) has been obtained from (\ref{Wi1}) with the aid of (\ref{hatg}), applied to $R_{\mathbf A^+}\boldsymbol\lambda_{\mathbf A}^+$ in place of $\mu$, and (\ref{repp}), the latter now with $X=D$ and $\kappa=g$.)

If Case~II holds, then for every $i\in I$ we also get from (\ref{Wi1}) and (\ref{wpot'})
\[W_{\alpha,\mathbf f}^{\boldsymbol\lambda_{\mathbf A},i}(\cdot)=s_i\kappa_\alpha\bigl(\cdot,(R_{\mathbf A}\boldsymbol\lambda_{\mathbf A}^++\zeta)-(R_{\mathbf A}\boldsymbol\lambda_{\mathbf A}^++\zeta)'\bigr)\text{ \ n.e.\ on \ }\mathbb R^n.\]
By \cite[Corollary~3.14]{FZ}, the function on the right (hence, also that on the left) in this relation takes the value $0$ at every $\alpha$-regular point of $A_p$, which gives~(\ref{reprrr1'}).

By Theorem~\ref{th-rel}, what has been shown just above yields that Theorem~\ref{desc-pot} will be proved once the following theorem has been established.

\begin{theorem}\label{desc-pot-g} Under the hypotheses of Theorem\/~{\rm\ref{desc-pot}} the following two assertions are equivalent for any\/ $\boldsymbol\lambda\in\mathcal E^{\boldsymbol\xi^+}_{g,\mathbf f^+}(\mathbf A^+,\mathbf a^+;D)${\rm:}
\begin{itemize}
\item[{\rm(i$'$)}] $\boldsymbol\lambda\in\mathfrak S^{\boldsymbol{\xi}^+}_{g,\mathbf f^+}(\mathbf A^+,\mathbf a^+;D)$.
\item[{\rm(ii$'$)}] There is a vector\/ $(c_j)_{j\in I^+}\in\mathbb R^{p-1}$ such that for all\/ $j\in I^+$
\begin{align}\label{b1'}W^{\boldsymbol\lambda,j}_{g,\mathbf f^+}&\geqslant c_j\quad(\xi^j-\lambda^j)\text{-a.e.},\\
\label{b2'}W^{\boldsymbol\lambda,j}_{g,\mathbf f^+}&\leqslant c_j\quad\lambda^j\text{-a.e.}
\end{align}
\end{itemize}
\end{theorem}

\begin{proof} Suppose first that (i$'$) holds. To verify (ii$'$), fix $j\in I^+$. For every $\boldsymbol\mu=(\mu^\ell)_{\ell\in I^+}\in\mathcal E^{\boldsymbol\xi^+}_{g,\mathbf f^+}(\mathbf A^+,\mathbf a^+;D)$ write $\boldsymbol\mu_j:=(\mu_j^\ell)_{\ell\in I^+}$ where $\mu_j^\ell:=\mu^\ell$ for all $\ell\ne j$ and $\mu_j^j=0$; then $\boldsymbol\mu_j\in\mathcal E^+_{g,\mathbf f^+}(\mathbf A^+)$. Also define $\tilde{f}_j:=f_j|_D+g^{\boldsymbol\lambda_j,j}$. By substituting (\ref{potv}) with $\kappa=g$ we then obtain
\begin{equation}\label{ftilde}\tilde{f}_j=f_j|_D+\sum_{\ell\in I^+, \ \ell\ne j}\,g(\cdot,\lambda^\ell).\end{equation}
Since $g(\cdot,\lambda^\ell)>0$ on $D$ for all $\ell\in I^+$ according to \cite[Lemma~4.1]{FZ} and since $f_j$ is lower bounded on $A_j$ by assumption, the function
\begin{equation}\label{sW}W_{g,\tilde{f}_j}^{\lambda^j}:=g(\cdot,\lambda^j)+\tilde{f}_j,\quad j\in I^+,\end{equation}
is likewise lower bounded on $A_j$. Furthermore, both $\tilde{f}_j$ and $W_{g,\tilde{f}_j}^{\lambda^j}$ are finite n.e.\ on the set $A_j^\circ$, which is clear from (\ref{Acirc}) and Lemma~\ref{l-Rpot}.

Applying (\ref{Re}) and (\ref{wen'}) we get for any
$\boldsymbol\mu\in\mathcal E^{\boldsymbol\xi^+}_{g,\mathbf f^+}(\mathbf A^+,\mathbf a^+;D)$ with the additional property that $\boldsymbol\mu_j=\boldsymbol\lambda_j$ (in particular for $\boldsymbol\mu=\boldsymbol\lambda$)
\[G_{g,\mathbf f^+}(\boldsymbol\mu)=G_{g,\mathbf f^+}(\boldsymbol\lambda_j)+G_{g,\tilde{f}_j}(\mu^j).\]
Combined with $G_{g,\mathbf f^+}(\boldsymbol\mu)\geqslant G_{g,\mathbf f^+}(\boldsymbol\lambda)$, this yields $G_{g,\tilde{f}_j}(\mu^j)\geqslant G_{g,\tilde{f}_j}(\lambda^j)$. Hence, $\lambda^j$ minimizes $G_{g,\tilde{f}_j}(\nu)$ where $\nu$ ranges over $\mathcal E^{\xi^j}_{g,\tilde{f}_j}(A_j,a_j;D)$. This enables us to show that there is $c_j\in\mathbb R$ such that
\begin{align}\label{sing1}W_{g,\tilde{f}_j}^{\lambda^j}&\geqslant c_j\quad(\xi^j-\lambda^j)\text{-a.e.},\\
\label{sing2}W_{g,\tilde{f}_j}^{\lambda^j}&\leqslant c_j\quad\lambda^j\text{-a.e.}\end{align}
In doing this we shall use permanently the fact that both $\xi^j$ and $\lambda^j$ have finite $\alpha$-Riesz energy are hence they are $c_\alpha$-absolutely continuous.

Indeed, (\ref{sing1}) holds with
\[c_j:=L_j:=\sup\,\bigl\{t\in\mathbb R: \ W_{g,\tilde{f}_j}^{\lambda^j}\geqslant t\quad(\xi^j-\lambda^j)\text{-a.e.}\bigr\}.\]
In turn, (\ref{sing1}) with $c_j=L_j$ implies $L_j<\infty$, for $W_{g,\tilde{f}_j}^{\lambda^j}<\infty$ n.e.\ on $A_j^\circ$, and hence $(\xi^j-\lambda^j)$-a.e.\ on $A_j$ by (\ref{acirc}). Also $L_j>-\infty$, because $W_{g,\tilde{f}_j}^{\lambda^j}$ is lower bounded on $A_j$ (see above).

We next establish (\ref{sing2}) with $c_j=L_j$. To this end, write for any $w\in\mathbb R$
\[A_j^+(w):=\bigl\{x\in A_j:\ W_{g,\tilde{f}_j}^{\lambda^j}(x)>w\bigr\},\quad A_j^-(w):=\bigl\{x\in A_j:\ W_{g,\tilde{f}_j}^{\lambda^j}(x)<w\bigr\}.\]
On the contrary, let (\ref{sing2}) with $c_j=L_j$ do not hold, i.e.\ $\lambda^j\bigl(A_j^+(L_j)\bigr)>0$. Since $W_{g,\tilde{f}_j}^{\lambda^j}$ is $\lambda^j$-measurable, one can choose $w_j\in(L_j,\infty)$ so that $\lambda^j\bigl(A_j^+(w_j)\bigr)>0$.
At the same time, as $w_j>L_j$, (\ref{sing1}) with $c_j=L_j$  yields $(\xi^j-\lambda^j)\bigl(A_j^-(w_j)\bigr)>0$.
Therefore, there exist compact sets $K_1\subset A_j^+(w_j)$ and $K_2\subset A_j^-(w_j)$ such that
\begin{equation}\label{c01}0<\lambda^j(K_1)<(\xi^j-\lambda^j)(K_2).\end{equation}

Write $\tau^j:=(\xi^j-\lambda^j)|_{K_2}$; then $g(\tau^j,\tau^j)<\kappa_\alpha(\tau^j,\tau^j)<\infty$, where the former inequality holds by (\ref{g-ineq}). As $\bigl\langle W_{g,\tilde{f}_j}^{\lambda^j},\tau^j\bigr\rangle\leqslant w_j\tau^j(K_2)<\infty$, we therefore get $\langle\tilde{f}_j,\tau^j\rangle<\infty$. Define
$\theta^j:=\lambda^j-\lambda^j|_{K_1}+b_j\tau^j$, where $b_j:=\lambda^j(K_1)/\tau^j(K_2)\in(0,1)$ by (\ref{c01}). Straightforward verification then shows that $\theta^j(A_j)=a_i$ and $\theta^j\leqslant\xi^j$, and hence $\theta^j\in\mathcal E^{\xi^j}_{g,\tilde{f}_j}(A_j,a_j;D)$. On the other hand,
\begin{align*}
\bigl\langle W_{g,\tilde{f}_j}^{\lambda^j},\theta^j-\lambda^j\bigr\rangle&=\bigl\langle
W_{g,\tilde{f}_j}^{\lambda^j}-w_j,\theta^j-\lambda^j\bigr\rangle\\&{}=-\bigl\langle
W_{g,\tilde{f}_j}^{\lambda^j}-w_j,\lambda^j|_{K_1}\bigr\rangle+b_j\bigl\langle
W_{g,\tilde{f}_j}^{\lambda^j}-w_j,\tau^j\bigr\rangle<0,\end{align*} which is
impossible by Lemma~\ref{aux43} with the (convex) set $\mathfrak E=\mathcal E^{\xi^j}_{g,\tilde{f}_j}(A_j,a_j;D)$. This contradiction establishes (\ref{sing2}).

Substituting (\ref{ftilde}) into (\ref{sW}) and then comparing the result obtained with (\ref{potv}) and (\ref{wpot'}), we see that
\begin{equation}\label{ww}W_{g,\tilde{f}_j}^{\lambda^j}=W_{g,\mathbf f^+}^{\boldsymbol\lambda,j}.\end{equation}
Combined with  (\ref{sing1}) and (\ref{sing2}), this establishes (\ref{b1'}) and
(\ref{b2'}), thus completing the proof that (i$'$) implies~(ii$'$).

Conversely, let (ii$'$) hold. On account of (\ref{ww}), for every $j\in I^+$ relations (\ref{sing1}) and (\ref{sing2}) are then fulfilled with $\tilde{f}_j$ defined by (\ref{ftilde}). This yields
\[\lambda^j\bigl(A_j^+(c_j)\bigr)=0\text{ \ and \ }\bigl(\xi^j-\lambda^j\bigr)\bigl(A_j^-(c_j)\bigr)=0.\]
For any $\boldsymbol\nu\in\mathcal E^{\boldsymbol\xi^+}_{g,\mathbf f^+}(\mathbf A^+,\mathbf a^+;D)$ we therefore get
\begin{align*}\bigl\langle
W^{\boldsymbol\lambda,j}_{g,\mathbf f^+},\nu^j-\lambda^j\bigr\rangle&=\bigl\langle
W_{g,\tilde{f}_j}^{\lambda^j}-c_j,\nu^j-\lambda^j\bigr\rangle\\
&{}=\bigl\langle W_{g,\tilde{f}_j}^{\lambda^j}-c_j,\nu^j|_{A_j^+(c_j)}\bigr\rangle+\bigl\langle
W_{g,\tilde{f}_j}^{\lambda^j}-c_j,(\nu^j-\xi^j)|_{A_j^-(c_j)}\bigr\rangle\geqslant0.\end{align*}
Summing up these inequalities over all $j\in I^+$, we conclude from Lemma~\ref{aux43} with the (convex) set $\mathfrak E=\mathcal E^{\boldsymbol\xi^+}_{g,\mathbf f^+}(\mathbf A^+,\mathbf a^+;D)$ that $\boldsymbol\lambda$ satisfies~(i$'$).\end{proof}

\subsection{Proof of Theorem~\ref{desc-sup}} For any $x\in D$ consider the inverse $K_x$ of $C\ell_{\,\overline{\mathbb R^n}}A_p$ relative to $S(x,1)$, $\overline{\mathbb R^n}$ being the one-point compactification of $\mathbb R^n$. Since $K_x$ is compact, there exists the (unique) $\kappa_\alpha$-equ\-il\-ibrium measure $\gamma_x\in\mathcal E^+_\alpha(K_x;\mathbb R^n)$ on $K_x$ possessing the properties $\|\gamma_x\|^2_\alpha=\gamma_x(K_x)=c_\alpha(K_x)$,
\begin{equation}\label{Keq}\kappa_\alpha(\cdot,\gamma_x)=1\text{ \ n.e.\ on \ }K_x,\end{equation}
and $\kappa_\alpha(\cdot,\gamma_x)\leqslant1$ on $\mathbb R^n$.  Note that $\gamma_x\ne0$, for $c_\alpha(K_x)>0$ in consequence of $c_\alpha(A_p)>0$ (see \cite[Chapter~IV, Section~5, n$^\circ$\,19]{L}). We assert that, under the stated requirements,
\begin{equation}\label{eq-desc}
S^{\gamma_x}_{\mathbb R^n}=\left\{
\begin{array}{lll} \breve{K}_x & \text{ \ if \ } & \alpha<2,\\ \partial_{\mathbb R^n} K_x  & \text{ \ if \ } & \alpha=2.\\ \end{array} \right.
\end{equation}
The latter equality in (\ref{eq-desc}) follows from \cite[Chapter~II, Section~3, n$^\circ$\,13]{L}. To establish the former equality,\footnote{We have brought here this proof, since we did not find a reference for this possibly known assertion.} we first note that $S^{\gamma_x}_{\mathbb R^n}\subset\breve{K}_x$ by the $c_\alpha$-ab\-sol\-ute continuity of $\gamma_x$. As for the converse inclusion, assume on the contrary that there is $x_0\in\breve{K}_x$ such that $x_0\notin S^{\gamma_x}_{\mathbb R^n}$. Choose $r>0$ with the property $\overline{B}(x_0,r)\cap S^{\gamma_x}_{\mathbb R^n}=\varnothing$. But $c_\alpha\bigl(B(x_0,r)\cap\breve{K}_x\bigr)>0$, hence by (\ref{Keq}) there exists $y\in B(x_0,r)$ such that $\kappa_\alpha(y,\gamma_x)=1$. The function $\kappa_\alpha(\cdot,\gamma_x)$ is $\alpha$-har\-monic on $B(x_0,r)$ \cite[Chapter~I, Section~5, n$^\circ$\,20]{L},
continuous on $\overline{B}(x_0,r)$, and takes at $y\in B(x_0,r)$ its maximum value $1$. Applying \cite[Theorem~1.28]{L} we obtain $\kappa_\alpha(\cdot,\gamma_x)=1$ $m_n$-a.e.\ on $\mathbb R^n$, hence everywhere on $\breve{K}^c_x$ by the continuity of $\kappa_\alpha(\cdot,\gamma_x)$ on $\bigl(S^{\gamma_x}_{\mathbb R^n}\bigr)^c$ \ $\bigl[{}\supset(\breve{K}_x)^c\bigr]$, and altogether n.e.\ on $\mathbb R^n$ by (\ref{Keq}). This means that $\gamma_x$ serves as the $\alpha$-Riesz equilibrium measure on the whole of $\mathbb R^n$, which is impossible.

Based on (\ref{reprrr}) and on the integral representation (\ref{int-repr}), we then arrive at the claimed relation (\ref{lemma-desc-riesz}) in view of the fact that, for every $x\in D$, $\varepsilon_x'$ is the Kelvin transform of the equilibrium measure $\gamma_x$ \cite[Section~3.3]{FZ}.

\subsection{Proof of Theorem \ref{zone}} Combining (\ref{reprrr1'}), (\ref{potl}) and (\ref{Wi}) gives (\ref{th}). Substituting the first relation from (\ref{th}) into (\ref{b2}) shows that under the stated assumptions the number $c_1$ from Theorem~\ref{desc-pot} is ${}>0$, while (\ref{b1}) now takes the (equivalent) form
\begin{equation}\label{err}\kappa_\alpha(\cdot,\lambda)\geqslant c_1>0\quad (\xi^1-\lambda^+)\text{-a.e.}\end{equation}
Having rewritten (\ref{b2}) as
\[\kappa_\alpha(\cdot,\lambda^+)\leqslant\kappa_\alpha(\cdot,\lambda^-)+c_1\text{ \ $\lambda^+$-a.e.},\]
we infer from \cite[Theorems~1.27, 1.29, 1.30]{L} that the same inequality holds on all of $\mathbb R^n$, which amounts to (\ref{b11}). In turn, (\ref{b11}) yields (\ref{b21}) when combined with~(\ref{err}). It follows directly from Theorem~\ref{desc-pot} that (\ref{b21}) and (\ref{b11}) together with the relation $\kappa_\alpha(\cdot,\lambda)=0$ n.e.\ on $D^c$ determine uniquely the solution $\lambda$ to the problem (\ref{sc-pr}) among the admissible measures.

Assume now that $\kappa_\alpha(\cdot,\xi^1)$ is continuous on $D$. Then so is $\kappa_\alpha(\cdot,\lambda^+)$. Indeed, since $\kappa_\alpha(\cdot,\lambda^+)$ is l.s.c.\ and since $\kappa_\alpha(\cdot,\lambda^+)=\kappa_\alpha(\cdot,\xi^1)-\kappa_\alpha(\cdot,\xi^1-\lambda^+)$ with $\kappa_\alpha(\cdot,\xi^1)$ continuous and
$\kappa_\alpha(\cdot,\xi^1-\lambda^+)$ l.s.c., it follows that $\kappa_\alpha(\cdot,\lambda^+)$ is also upper semicontinuous, hence continuous. Therefore, by the continuity of $\kappa_\alpha(\cdot,\lambda^+)$ on $D$, (\ref{b21}) implies (\ref{b21'}). Thus, by (\ref{th}) and (\ref{b21'}),
\[g(\cdot,\lambda^+)=c_1\text{ \ on \ }S_D^{\xi^1-\lambda^+},\]
which implies (\ref{cg}) in view of \cite[Lemma 3.2.2]{F1} (with $X=D$ and $\kappa=g$).

Omitting now the requirement of continuity of $\kappa_\alpha(\cdot,\xi^1)$, assume further that $\alpha<2$ and $m_n(D^c)>0$. If on the contrary (\ref{pg}) is not fulfilled, then there is $x_0\in S_D^{\xi^1}$ such that $x_0\notin S_D^{\lambda^+}$. Hence one can choose $r>0$ so that
\begin{equation}\label{ball}\overline{B}(x_0,r)\subset D\quad\text{and}\quad\overline{B}(x_0,r)\cap S^{\lambda^+}_D=\varnothing.\end{equation}
Then $(\xi^1-\lambda^+)\bigl(\overline{B}(x_0,r)\bigr)>0$, and hence there exists $y\in\overline{B}(x_0,r)$ with the property
$\kappa_\alpha(y,\lambda)=c_1$ (cf.\ (\ref{b21})), or equivalently
\begin{equation}\label{eq;eq}\kappa_\alpha(y,\lambda^+)=\kappa_\alpha(y,\lambda^-)+c_1.\end{equation}
Since $\kappa_\alpha(\cdot,\lambda^+)$ is $\alpha$-harmonic on $B(x_0,r)$ and continuous on $\overline{B}(x_0,r)$ and since
$\kappa_\alpha(\cdot,\lambda^-)+c_1$ is $\alpha$-super\-harmonic on $\mathbb R^n$, we conclude from (\ref{b11}) and (\ref{eq;eq}) with the aid of
\cite[Theorem~1.28]{L} that
\begin{equation}\label{contr}\kappa_\alpha(\cdot,\lambda)=\kappa_\alpha\bigl(\cdot,\lambda^-)+c_1\quad m_n\mbox{-a.e.\ on \ }\mathbb R^n.\end{equation}
This implies $c_1=0$, because by (\ref{bal-eq}) and (\ref{reprrr})
\[\kappa_\alpha(\cdot,\lambda^+)=\kappa_\alpha\bigl(\cdot,(\lambda^+)'\bigr)=\kappa_\alpha(\cdot,\lambda^-)\text{ \ n.e.\ on \ }D^c,\]
hence $m_n$-a.e.\ on $D^c$. A contradiction.

Similar arguments enable us to establish (\ref{supp}). Indeed, if (\ref{supp}) were not fulfilled at some $x_1\in D\setminus S_D^{\lambda^+}$, then (\ref{eq;eq}) would hold with $x_1$ in place of $y$ (cf.\ (\ref{b11})) and, furthermore, one could choose $r>0$ so that
(\ref{ball}) would be valid with $x_1$ in place of $x_0$. Therefore, since $\kappa_\alpha(\cdot,\lambda^+)$ is $\alpha$-harmonic on $B(x_1,r)$ and continuous on $\overline{B}(x_1,r)$ and since $\kappa_\alpha(\cdot,\lambda^-)+c_1$ is $\alpha$-super\-harmonic on $\mathbb R^n$, we would arrive again at (\ref{contr}) and hence at $c_1=0$, which is impossible.

\section{Examples}

The purpose of the examples below is to illustrate the assertions from Section~\ref{form:main}. Note that in either Example~\ref{ex} or Example~\ref{ex2} the set $A_2=D^c$ is not $\alpha$-thin at infinity.

\begin{example}\label{ex}{\rm Let $n\geqslant3$, $\alpha<2$, $D=B_r:=B(0,r)$, where $r\in(0,\infty)$, and let $I^+=\{1\}$, $A_1=D$, $\mathbf a=\mathbf 1$, $\mathbf f=\mathbf 0$. Define $\xi^1:=q\lambda_r$, where $q\in(1,\infty)$ and $\lambda_r$ is the $\kappa_\alpha$-capacitary measure on $\overline{B}_r:=\overline{B}(0,r)$ (Remark~\ref{remark}). As follows from \cite[Chapter~II, Section~3, n$^\circ$\,13]{L}, $\xi^1\in\mathcal E_\alpha^+(A_1,q;\mathbb R^n)$, $S_D^{\xi^1}=D$ and $\kappa_\alpha(\cdot,\xi^1)$ is continuous on $\mathbb R^n$. Since $\mathbf f=\mathbf 0$, Problem~\ref{pr2gen} reduces to the problem (\ref{sc-pr}) of minimizing $\kappa_\alpha(\mu,\mu)$ over all $\mu\in\mathcal E_\alpha(\mathbb R^n)$ such that $\mu^+\in\mathcal E^{\xi^1}_\alpha(A_1,1;\mathbb R^n)$ and $\mu^-\in\mathcal E^+_\alpha(A_2,1;\mathbb R^n)$, which by Theorem~\ref{th-rel} is equivalent to the problem of minimizing $g^\alpha_D(\nu,\nu)$ over $\mathcal E^{\xi^1}_{g^\alpha_D}(A_1,1;D)$. According to Theorems~\ref{th-suff}, \ref{th-rel} and Corollary~\ref{cor-unique}, these two constrained minimum energy problems are uniquely solvable (no short-circuit occurs between $D$ and $D^c$), and their solutions, denoted respectively by $\lambda_{\alpha,\mathbf A}=\lambda^+-\lambda^-$ and $\lambda_{g,A_1}$, are related to each other as follows:
\[\lambda_{\alpha,\mathbf A}=\lambda_{g,A_1}-\lambda_{g,A_1}'.\]
Furthermore, by (\ref{lemma-desc-riesz}), (\ref{cg}) and (\ref{pg}),
\[S_D^{\lambda^+}=S_D^{\lambda_{g,A_1}}=S_D^{\xi^1}=D,\quad S_{\mathbb R^n}^{\lambda^-}=D^c,\]
\begin{equation}\label{hryu}c_{g_D^\alpha}\bigl(S_D^{\xi^1-\lambda^+}\bigr)<\infty,\end{equation}
while by (\ref{th}), (\ref{b11}) and (\ref{b21'}),
\begin{equation}\label{det}\kappa_\alpha(\cdot,\lambda_{\alpha,\mathbf A})=\left\{
\begin{array}{lll} c_1 & \text{ \ on \ } & S^{\xi^1-\lambda^+}_D,\\ 0  & \text{ \ on \ } & D^c,\\ \end{array} \right.\end{equation}
\begin{equation}\label{ex1}\kappa_\alpha(\cdot,\lambda_{\alpha,\mathbf A})\leqslant c_1\text{ \ on \ }D\setminus S^{\xi^1-\lambda^+}_D,\end{equation}
where $c_1>0$. (In (\ref{det}) we have used the fact that for the given $\alpha$ and $D$, $I_{\alpha,D^c}=\varnothing$.)
Moreover, according to Theorem~\ref{desc-pot}, relations (\ref{det}) and (\ref{ex1}) determine uniquely the solution $\lambda_{\alpha,\mathbf A}$ among the admissible measures.}
\end{example}

\begin{example}\label{ex2}{\rm Let $n=3$, $\alpha=2$, $\mathbf f=\mathbf 0$, $\mathbf a=\mathbf 1$. Define $D:=\{x=(x_1,x_2,x_3)\in\mathbb R^3: \  x_1>0\}$ and $A_1:=\sum_{k\in\mathbb N}\,K_k$, where
\[K_k:=\bigl\{(x_1,x_2,x_3)\in D: \ x_1=k^{-1}, \ x_2^2+x_3^2\leqslant k^2\bigr\},\quad k\in\mathbb N.\]
Let $\lambda_k$ be the $\kappa_2$-capacitary measure on $K_k$ (Remark~\ref{remark}); hence $\lambda_k(K_k)=1$ and $\|\lambda_k\|^2_2=\pi^2/(2k)$ by \cite[Chapter~II, Section~3, n$^\circ$\,14]{L}. Define
\[\xi^1:=\sum_{k\in\mathbb N}\,k^{-2}\lambda_k.\]
In the same manner as in the proof of Theorem~\ref{th-unsuff} one can see that $\xi^1$ is a positive Radon measure carried by $A_1$ with $\kappa_2(\xi^1,\xi^1)<\infty$ and $\xi^1(A_1)\in(1,\infty)$. By Theorem~\ref{th-suff}, Problem~\ref{pr2gen} for the constraint $\boldsymbol\xi=(\xi^1,\infty)$ and the condenser $\mathbf A=(A_1,D^c)$ has therefore a (unique) solution $\lambda_{\alpha,\mathbf A}=\lambda^+-\lambda^-$ (no short-circuit occurs between $A_1$ and $D^c$), although $D^c\cap C\ell_{\mathbb R^3}A_1=\partial D$ and hence
\[c_2\bigl(D^c\cap{C\ell}_{\mathbb R^3}A_1\bigr)=\infty.\]
Furthermore, since each $\kappa_2(\cdot,\lambda_k)$, $k\in\mathbb N$, is continuous on $\mathbb R^n$ and bounded from above by $\pi^2/(2k)$, the potential $\kappa_2(\cdot,\xi^1)$ is continuous on $\mathbb R^n$ by the uniform convergence of the sequence $\sum_{k\in\mathbb N}\,k^{-2}\kappa_2(\cdot,\lambda_k)$.
Hence (\ref{hryu}), (\ref{det}) and (\ref{ex1}) also hold in the present case with $\alpha=2$, again with $c_1>0$, and relations (\ref{det}) and (\ref{ex1}) determine uniquely the solution $\lambda_{\alpha,\mathbf A}$ among the admissible measures. Also note that $S^{\lambda^-}_{\mathbb R^n}=\partial D$ according to~(\ref{lemma-desc-riesz}).}\end{example}


\begin{thebibliography}{99}

\bibitem{BH} J.~Bliednter and W.~Hansen, Potential Theory: An Analytic and Probabilistic Approach to Balayage, Springer, Berlin, 1986.

\bibitem{B2} N.~Bourbaki, Int\'egration. Chapitres~I--IV, Actualit\'es Sci. Ind.
1175, Paris, 1952.

\bibitem{Bou} N.~Bourbaki, Int\'egration. Int\'egration des Mesures. Chapitre~V, Hermann, Paris, 1956.

\bibitem{B1}
N.~Bourbaki, General Topology. Chapters I--IV,
Springer, Berlin, 1989.

\bibitem{Brelo2} M.~Brelot, On Topologies and Boundaries in
Potential Theory, Lecture Notes in Math. 175, Springer,
Berlin, 1971.

\bibitem{Car} H.~Cartan, Th\'eorie du potentiel Newtonien:
\'energie, capacit\'e, suites de potentiels, Bull.\ Soc.\ Math.\ France
{\bf 73} (1945), 74--106.

\bibitem{CD} H.~Cartan and J.~Deny, Le principe du maximum en th\'eorie du potentiel et la notion de fonction surharmonique, Acta Sci.\ Math.\ (Szeged) {\bf 12} (1950), 81--100.

\bibitem{D1} J.~Deny, Les potentiels d'\'energie finie, Acta Math.\ {\bf 82} (1950), 107--183.

\bibitem{D2} J.~Deny, Sur la d\'{e}finition de
l'\'{e}nergie en th\'{e}orie du potentiel, Ann.\ Inst.\ Fourier (Grenoble) {\bf 2} (1950), 83--99.

\bibitem{DFHSZ} P.D.~Dragnev, B.~Fuglede, D.P.~Hardin, E.B.~Saff and N.~Zorii, Minimum Riesz energy problems for a condenser with touching plates, Potential Anal.\ {\bf 44} (2016), 543--577.

\bibitem{DFHSZ2} P.D.~Dragnev, B.~Fuglede, D.P.~Hardin, E.B.~Saff and N.~Zorii, Constrained minimum Riesz energy problems for a condenser with intersecting plates, arXiv:1710.01950v1 (2017), 33~p.

\bibitem{E2} R.~Edwards, Functional Analysis. Theory and
Applications, Holt, Rinehart and Winston, New York, 1965.

\bibitem{F1} B.~Fuglede, On the theory of potentials in
locally compact spaces, Acta Math.\ {\bf 103} (1960), 139--215.

\bibitem{Fu4} B.~Fuglede, Capacity as a sublinear functional generalizing an integral, Mat.-Fys.\ Medd.\ Danske Vid.\ Selsk.\ {\bf 38} (1971), no.~7, 44~pp.

\bibitem{Fu5} B.~Fuglede, Symmetric function kernels and sweeping of measures, Analysis Math.\ {\bf 42} (2016), 225--259.

\bibitem{FZ} B.~Fuglede and N.~Zorii, Green kernels associated with Riesz kernels, Ann.\ Acad.\ Sci.\ Fenn.\ Math. {\bf 43} (2018), 121--145.

\bibitem{Gauss} C.F.~Gauss,  Allgemeine Lehrs\"atze in Beziehung auf die im
verkehrten Verh\"altnisse des Quadrats der Entfernung wirkenden
Anziehungs-- und Absto{\ss}ungs--Kr\"afte (1839), Werke {\bf 5} (1867),
197--244.

\bibitem{HWZ}
H.~Harbrecht, W.L.~Wendland and N.~Zorii, Riesz minimal
energy problems on $C^{k-1,k}$-manifolds,  Math.\ Nachr.\ \textbf{287} (2014), 48--69.

\bibitem{L}
N.S.~Landkof, Foundations of Modern Potential Theory, Springer,
Berlin, 1972.

\bibitem{OWZ}
G.~Of, W.L.~Wendland and  N.~Zorii, On the numerical
solution of minimal energy problems, Complex Var.\
Elliptic Equ.\ \textbf{55} (2010), 991--1012.

\bibitem{O}
M.~Ohtsuka, On potentials in locally compact spaces,
J.~Sci.\ Hiroshima Univ.\ Ser.~A-1, {\bf 25} (1961), 135--352.


\bibitem{ZUmzh} N.~Zorii, A noncompact variational problem in
Riesz potential theory,~I; II,  Ukrain.\ Math.~J. {\bf 47} (1995), 1541--1553; {\bf 48} (1996), 671--682.

\bibitem{ZPot1} N.~Zorii, Interior capacities of condensers in
locally compact spaces, Potential Anal. {\bf 35} (2011), 103--143.

\bibitem{Z9}
N.~Zorii, Constrained energy problems with external fields for vector measures,
 Math.\ Nachr.\ {\bf 285} (2012), 1144--1165.

\bibitem{ZPot2} N.~Zorii, Equilibrium problems for infinite dimensional vector potentials with
external fields, Potential Anal.\ {\bf 38} (2013), 397--432.

\bibitem{ZPot3} N.~Zorii, Necessary and sufficient conditions for the solvability of the Gauss variational problem for infinite dimensional vector measures, Potential Anal.\
{\bf 41} (2014), 81--115.


\end{thebibliography}
\end{document}